\documentclass[Threeside,10pt]{article}
\usepackage [latin1]{inputenc}
\usepackage[english]{babel}
\usepackage{amsmath}
\usepackage{amsthm}
\usepackage{algorithm}
\usepackage{cite}
\usepackage{amsfonts}
\usepackage{amssymb}
\usepackage{mathrsfs}
\usepackage{graphicx}
\usepackage{colortbl,dcolumn}
\usepackage{marvosym}
\usepackage{ifsym}
\usepackage{paralist}
\usepackage{xcolor}
\usepackage[
colorlinks,
citecolor=blue,
linkcolor=blue,
backref=page
allcolors=black
]{hyperref}
\allowdisplaybreaks
\def\e{\varepsilon}
       \newtheorem{lemma}{\bf Lemma}[section]
       \newtheorem{theorem}{\bf Theorem}[section]
       
       \newtheorem{corollary}{\bf Corollary}[section]
       \newtheorem{definition}{\bf Definition}[section]
       \newtheorem{remark}{\bf Remark}[section]

       \numberwithin{equation}{section}

\begin{document}
\title{\textbf{Moser's Theorem with Frequency-preserving}}
\author{Chang Liu $^{1}$, Zhicheng Tong $^{*,2}$, Yong Li $^{1,2}$ }

\renewcommand{\thefootnote}{}
\footnotetext{\hspace*{-6mm}
\begin{tabular}{l l}
 $^{*}$~~Corresponding author. E-mail addresses: tongzc20@mails.jlu.edu.cn\\
 $^{1}$~~~School of Mathematics and Statistics, Center for Mathematics and Interdisciplinary Sciences,\\
 ~~~ Northeast Normal University, Changchun 130024, People's Republic of China.\\
$^{2}$~~~School of Mathematics, Jilin University, Changchun 130012, People's Republic of China.
\end{tabular}}

\date{}
\maketitle

\begin{abstract}

This paper mainly concerns the KAM persistence of the mapping $\mathscr{F}:\mathbb{T}^{n}\times E\rightarrow \mathbb{T}^{n}\times \mathbb{R}^{n}$ with intersection property, where $E\subset \mathbb{R}^{n}$ is a connected closed bounded domain with interior points.  By assuming that the frequency mapping satisfies certain topological degree condition and weak convexity condition, we prove some Moser type results about the invariant torus of mapping $\mathscr{F}$ with frequency-preserving under small perturbations. To our knowledge, this is the first approach to Moser's theorem with frequency-preserving. Moreover, given perturbed mappings over $ \mathbb{T}^n $, it is shown that such persistence still holds when the frequency mapping and perturbations are only continuous about parameter  beyond Lipschitz or even H\"older type. We also touch the parameter without dimension limitation problem under such settings. 
\end{abstract}

\section{Introduction}
\setcounter{equation}{0}
 The celebrated  KAM theory is established by Kolmogorov \cite{Ko}, Arnold \cite{Ar} and Moser \cite{Mo}. It mainly concerns the stability of motions or orbits in dynamical systems under small perturbations and indeed has a long history of over sixty years. So far, as well known, KAM theory has been widely spread to various  systems, such as volume-preserving flows due to Broer et al  \cite{Br}, generalized Hamiltonian systems due to Parasyuk \cite{Pa} and Li and Yi \cite{L1,L2}, finitely differentiable Hamiltonians due to Salamon \cite{Sa}, Bounemoura \cite{Ab} and Koudjinan \cite{CK}, Gevrey Hamiltonians due to Popov \cite{Po}, and multiscale dynamical systems due to Qian et al \cite{QW}.  For other related results, see for instance,  \cite{QL,CQ, MR1938331,MR1302146,MR1483538,MR1843664,MR1354540}. In studying discrete dynamical systems,    Moser \cite{Mo} firstly  studied the persistence of invariant tori of twist mappings with perturbation only on the action variable, namely of the following form
		\[	\mathcal{G}:
		\begin{cases}
			\theta^{1}=\theta+r,\\
			r^{1}=r+ g(\theta,r),
		\end{cases}\]
where $ g $ is assumed to be of class $ C^{333} $. He stated a differentiable invariant curve theorem, which also is of great importance for the study of stability of periodic solutions. An analytic invariant curve theorem was provided in \cite{Si} by Siegel and Moser. Concentrating on finitely differentiable case, Svanidze established a KAM theorem for twist mappings in \cite{SV}. The version of class $C^{\alpha}$ with $\alpha>4$ and the optimal situation of class $C^{\alpha}$ with $\alpha>3$ about mappings on the annulus were 
due to R\"ussmann \cite{Ru} and Herman \cite{He,H1}, respectively.  For relevant works on the existence of invariant tori of voluming-preserving mappings, see Cheng and Sun \cite{CS} $(3{\rm -dimensional})$, Feng and Shang \cite{FS}, Shang \cite{MR1790659} and Xia \cite{XZ} $(n\geq 3)$. Cong et al \cite{CL} gave the persistence of the invariant tori when considering mappings with the intersection property, which has different numbers of actions and angular variables.  Levi and Moser \cite{ML} gave a Lagrange proof of the invariant curve theorem for twist mappings using the method introduced by Moser \cite{M1}. The invariant curve theorem for quasi-periodic reversible mappings was studied by Liu \cite{LB}. For some reversible mappings, see Sevryuk's book \cite{Se}. Recently, Yang and Li \cite{YL}, Zhao and Li \cite{ZL}  have extended the existence of invariant tori to resonance surfaces of twist mappings and multiscale mappings, respectively. Apart from above, Liu and Xing \cite{LX} presented a new proof of Moser's theorem for twist mappings with a parameter.

Among the KAM theory for Hamiltonian systems, the preservation of prescribed frequency is also important in studying invariance of dynamics under small perturbations, see Salamon \cite{Sa} for instance, especially without certain nondegeneracy such as Kolmogorov or  R\"ussmann  conditions, see Du et al \cite{DL} and Tong et al \cite{TD}. However, as generally known, the frequency of dynamical systems may have a drift effect by the perturbations during the KAM iteration, and therefore 
it is indeed difficult to frequency-preserving. 
\textit{To the best of our knowledge, there are no KAM results for twist mappings on this aspect, no one knows whether the prescribed frequency could be preserved for an invariant torus.} In this paper, we will touch this question. To this end,  it is necessary to propose some transversality conditions involving topological degree condition as well as certain weak convexity condition to overcome the drift of frequency, see \cite{DL, TD} and the references therein. Based on them, we will establish the KAM persistence with frequency-preserving for twist mappings with  intersection property. The so-called intersection property is that any torus close to the invariant torus of the unperturbed system intersects its image under the mapping.  More precisely, denote by  $\mathbb{T}^{n}=\mathbb{R}^{n}/ 2\pi\mathbb{Z}^{n}$ the $n$-dimensional torus, and let $E\subset \mathbb{R}^{n}$ be a connected closed bounded domain with interior points. Then consider the twist  mapping $\mathscr{F}:\mathbb{T}^{n}\times E\rightarrow \mathbb{T}^{n}\times\mathbb{R}^{n}$ with intersection property
		\begin{equation}\label{equation-1}
			\mathscr{F}:
			\begin{cases}
				\theta^{1}=\theta+\omega(r)+\e f(\theta,r,\e),\\
				r^{1}=r+\e g(\theta,r,\e),
			\end{cases}
		\end{equation}
		where the perturbations $f$ and $g$ are real analytic about $(\theta,r)$ on $\mathbb{T}^{n}\times E$,  $\omega$ is  assumed to be only continuous about $r$ on $E$, and $ \varepsilon $ is a sufficiently small scalar.  By introducing parameter translation technique, we   prove in Theorem \ref{theorem-1} the persistence of invariant tori of  such a family of twist mappings with the frequency unchanged under small perturbations, and as a byproduct,  \textit{this gives rise to the first result for  Moser's theorem with frequency-preserving.}  Moreover, using similar approach, we also investigate the perturbed mapping $\mathscr{F}:\mathbb{T}^{n}\times \Lambda\rightarrow \mathbb{T}^{n}$  with parameter
			\begin{equation}\label{intro2}
				\mathscr{F}:	\theta^{1}=\theta+\omega(\xi)+\e f(\theta,\xi,\e),
			\end{equation}
			provided with $ \Lambda $ the same as $ E $, and $ \varepsilon $ is a sufficiently small scalar. The perturbation $ f $ is analytic about $ \theta $ on $ \mathbb{T}^n $, and \textit{only continuity} with respect to the parameter $ \xi \in \Lambda $ is assumed for $ f $ and $ \omega $. \textit{Under such weak settings, we show the unexpected frequency-preserving KAM persistence via transversality conditions in Theorem \ref{theorem-2}.} As an explicit example, one could deal with irregular perturbations, such as nowhere differentiable systems.  

 This paper is organized as follows. Section \ref{SEC2} introduces some basic notations on modulus of continuity. In Section \ref{SEC3}, we state Theorem \ref{theorem-1} with frequency-preserving for twist mapping \eqref{equation-1}   satisfying the intersection property. When $n=1$, we obtain Moser's invariant curve theorem with frequency-preserving, see Corollary \ref{cor-1}. Theorem \ref{theorem-2} concerns mapping \eqref{intro2} with only angular variables, and shows the persistence of invariant torus with frequency-preserving, where the perturbation $f(\theta,\xi,\e)$ is real analytic about $\theta$, continuous about the parameter $\xi$, and the frequency $\omega(\xi)$ is also continuous about $\xi$. To emphasize the weak regularity, we provide Corollary \ref{COROMH} with nowhere H\"older about parameter. Some discussions involving \textit{parameter without dimensional limitation problem} are also given in Section \ref{SEC3}. The one cycle of KAM steps from $\nu$-th to $(\nu+1)$-th step is shown in Section \ref{SEC4}. In more detail, instead of digging out a series of decreasing domains for frequency, we construct a translation  to keep frequency unchanged during the iterative process. In addition, we have to construct a conjugate mapping to overcome the loss of intersection property. Finally, Section \ref{SEC5} is devoted to the proof of our main results.

 \section{Preliminaries}\label{SEC2}
\setcounter{equation}{0}
To describe only continuity, we first introduce some definitions in this section, involving the modulus of continuity and the norm based on it.
\begin{definition}\label{def-1}
A modulus of continuity is denoted as $\varpi(x)$, which is a strictly monotonic increasing continuous function on $\mathbb{R}^{+}$ that satisfies
 \begin{equation*}
\lim_{x\rightarrow 0^{+}}\varpi(x)=0,
\end{equation*}
and
\begin{equation*}
\varlimsup_{x\rightarrow 0^{+}}\frac{x}{\varpi(x)}<+\infty.
\end{equation*}
\end{definition}
\begin{definition}
Let a modulus of continuity $ \varpi $ be given. A function $f(x)$ is said to be $\varpi$ continuous about $ x $, if
\begin{equation*}
|f(x)-f(y)|\leq \varpi(|x-y|),\qquad \forall\ 0<|x-y|\leq 1.
\end{equation*}
\end{definition}

It is well known that a mapping defined on a bounded connected closed set in a finite dimensional space must admit a modulus of continuity, see  \cite{Herman3,KO}. For example, for a function $ f(x) $ defined on $ [0,1]  \subset {\mathbb{R}^1} $, it automatically admits a modulus of continuity
\[{\varpi _{f }}\left( x \right): = \mathop {\sup }\limits_{y \in \left[ {0,1} \right],0 < \left| {x - y} \right| \leq 1 } \left| {f\left( x \right) - f\left( y \right)} \right|.\]
We therefore only concentrate on modulus of continuity throughout this paper, especially in Theorem \ref{theorem-2}.

Next we will introduce the comparison relation between the strength and the weakness of modulus of continuity.

\begin{definition}\label{DE2.3}
Assume $\varpi_{1}$, $\varpi_{2}$ are two modulus of continuity. We say $\varpi_{1}$ is to be not weaker than $\varpi_{2}$ if
\begin{equation*}
\varlimsup_{x\rightarrow 0^{+}}\frac{\varpi_{1}(x)}{\varpi_{2}(x)}<+\infty,
\end{equation*}
and denote it as $\varpi_{1}\leq \varpi_{2}$ $($or $\varpi_{2}\geq \varpi_{1})$.

\end{definition}
\begin{remark}\label{remark-1}
\begin{itemize}
\item[(a)]If the function $f$ is real analytic about $x$ on a bounded closed set, then $f$ is naturally continuously differentiable with $x$ of order one. Obviously, one has $|f(x)-f(y)|\leq c|x-y|$ for some $c>0$ independent of $x, y$, that is, there exists a modulus of continuity $\varpi_{1}(x)=x$ with
	\begin{equation*}
	|f(x)-f(y)|\leq c\varpi_{1}(|x-y|),\qquad \forall \ 0<|x-y|\leq 1.
\end{equation*}
\item[(b)]The classical $ \alpha $-H\"older case corresponds to modulus of continuity $ \varpi_{\rm H}^\alpha(x) = x^\alpha $ with some $ 0<\alpha<1 $, and the Logarithmic Lipschitz case $\varpi_{\rm {LL}}(x)\sim (-\log x)^{-1}$  as $ x \to 0^+ $ is weaker than arbitrary $ \alpha $-H\"{o}lder continuity, that is, $ \varpi_{\rm H}^\alpha(x) \leq \varpi_{\rm {LL}}(x) $. Both of them characterise regularity weaker than that of Lipschitz.
\end{itemize}
\end{remark}

It needs to be pointed out that, the regularity of the majority of functions is indeed very weak from the perspective of Baire category,  such as nowhere differentiable.  In fact, the nowhere differentiable regularity could be even worse. More precisely, we present the following theorem constructing very weak continuity, one can see details from Theorem 7.2 in \cite{TD}.

\begin{theorem}\label{nonowh}
	Given a modulus of continuity $ {\varpi _1} $, there exists a  function $ f $ (actually, a family) 	on $ \mathbb{R} $ and a modulus of continuity $ {\varpi _2} \geq  {\varpi _1}$, such that $ f $ is $ {\varpi _2} $ continuous, but nowhere $ {\varpi _1} $ continuous.
\end{theorem}
\begin{remark}
	These kind of functions are usually constructed by trigonometric series admitting self-similarity, similar to Weierstrass function and so on.
\end{remark}
\begin{remark}\label{noholder}
	As a direct application, we can construct a family of functions, which are nowhere Lipschitz or even nowhere H\"older continuous.
\end{remark}

Finally, in order to specify the norm based on the modulus of continuity, we need to give the domains of the variables in detail.  Throughout this paper, let
\begin{align*}
D(h)&:=\{\theta\in\mathbb{C}^{n}: {\rm Re}\ \theta\in\mathbb{T}^{n},\ |{\rm Im}\ \theta|\leq h\},\\
G(s)&:=\{r\in \mathbb{C}^{n}:{\rm Re}\ r\in E, \ | {\rm Im}\ r|\leq s\}
\end{align*}
be the complex neighborhoods of $\mathbb{T}^{n}$ and $E$ for given $h,s>0$. For each vector $r=(r_{1},\cdots,r_{n})\in\mathbb{R}^{n}$, we denote by $|r|$ the $l^{1}$-norm of $r$:
\begin{equation*}
|r|=|r_{1}|+\cdots+|r_{n}|.
\end{equation*}
Also, for ease of notation, we write
$\mathcal{D}(h,s):=D(h)\times G(s)$. Next we introduce the  norm defined as follows.
\begin{definition}
For the perturbation function $f(\theta,r,\e)$, which is real analytic about $(\theta,r)\in\mathcal{D}(h,s)$, one can also claim that $f(\theta,r,\e)$ is $\varpi_{1}(x)=x$ continuous about $r$ due to Remark \ref{remark-1}, define its norm as follows
\begin{equation*}
||f||_{\mathcal{D}(h,s)}:=|f|_{\mathcal{D}(h,s)}+[f]_{\varpi_{1}},
\end{equation*}
where
\begin{equation*}
|f|_{\mathcal{D}(h,s)}=\sup_{(\theta,r)\in \mathcal{D}(h,s)}|f(\theta,r)|,
\end{equation*}
and
\begin{equation*}
[f]_{\varpi_{1}}=\sup_{\theta\in D(h)}\sup_{\substack{r',r''\in G(s)\\0<|r'-r''|\leq 1}}\frac{|f(\theta,r',\e)-f(\theta,r'',\e)|}{\varpi_{1}(|r'-r''|)}.
\end{equation*}
\begin{remark}
As to weaker continuity described by certain modulus of continuity $ \varpi \geq \varpi_1$, one only needs to change $ \varpi_1 $ to $ \varpi $  in the norm accordingly.
\end{remark}

\end{definition}
\section{Main results}\label{SEC3}
This section is divided into two parts, namely stating our main KAM results as well as giving some further discussions.
\subsection{Frequency-preserving KAM}
\setcounter{equation}{0}
Before starting,  let us make some  preparations. Following Remark \ref{remark-1}, there exists a modulus of continuity $\varpi_{1}(x)=x$ such that $f$ and $g$ are automatically $\varpi_{1}$ continuous about $r$. Besides, the following assumptions are crucial to our KAM theorems.\\
 \begin{itemize}
 \item[{\rm (A1)}] Let $p\in \mathbb{R}^{n}$ be given in advance and denote by $E^{\circ}$ the interior of $E$. Assume that
 \begin{equation}\label{A1}
 \deg\left(\omega(\cdot),E^{\circ},p\right)\neq 0.
 \end{equation}
 \item[{\rm (A2)}] Assume that $\omega(r_{*})=p$ with some $r_{*}\in E^{\circ}$ by \eqref{A1}, and
 \begin{equation*}
 |\langle k,\omega(r_{*})\rangle -k_{0}|\geq \frac{\gamma}{|k|^{\tau}},\qquad \forall\ k\in \mathbb{Z}^{n}\backslash \{0\},\ k_{0}\in\mathbb{Z}, \ |k_{0}|\leq M_{0}|k|,
 \end{equation*} where $\gamma>0$, $\tau>n-1$ is fixed, and $M_{0}$ is assumed to be the upper bound of $|\omega|$ on $E$.\\
\item[{\rm (A3)}] Assume that $B(r_{*},\delta)\subset E^{\circ}$ with $ \delta>0 $ is a neighborhood of $r_{*}$. There exists a modulus of continuity $\varpi_{2}$ 
 such that
\begin{equation*}
|\omega(r')-\omega(r'')|\geq \varpi_{2}(|r'-r''|),\qquad r',r''\in B(r_{*},\delta), \quad 0<|r'-r''|\leq 1.
\end{equation*}
\end{itemize}

Via these assumptions, we are now in a position to present the following KAM theorem for twist mapping with intersection property, \textit{which is the first frequency-preserving result on Moser's theorem to the best of our knowledge.}

 \begin{theorem}\label{theorem-1} Consider mapping \eqref{equation-1} with intersection property. Assume that the perturbations are real analytic about $(\theta,r)$, and the frequency $\omega$ is  continuous about $r$. Moreover, ${\rm (A1)}$-${\rm (A3)}$ hold. Then there exists a sufficiently small $\e_{0}$, a transformation $\mathscr{W}$ when $0<\e<\e_{0}$. The transformation $\mathscr{W}$ is a conjugation from $\mathscr{F}$ to $\hat{\mathscr{F}}$, and $\hat{\mathscr{F}}(\theta,r)=(\theta+\omega(r_{*}),r-\tilde{r})$ is the integrable rotation on $\mathbb{T}^{n}\times E$ with frequency $\omega(r_{*} )=p$, where $\tilde{r}$ is the translation about the action  $r$ resulting from the transformation $\mathscr{W}$, and the constant $\tilde{r} \rightarrow 0$ as $\e\rightarrow0$. That is, the following  holds:
 \begin{equation*}
 \mathscr{W}\circ\hat{\mathscr{F}}=\mathscr{F}\circ\mathscr{W} .
 \end{equation*}
 \end{theorem}

 When $n=1$, consider the area-preserving mapping of the form \eqref{equation-1}, which obviously satisfies the intersection property. Correspondingly, we obtain Moser's invariant curve theorem with frequency-preserving as stated in the following corollary.
 \begin{corollary}\label{cor-1}
Consider mapping \eqref{equation-1} for $n=1$. Assume that the perturbations are real analytic about $(\theta,r)$, and the frequency $\omega$ is continuous and strictly monotonic concerning $r$. Moreover, $({\rm A2})$ and $({\rm A3})$ hold. Then there exists a sufficiently small $\e_{0}$, a transformation $\mathscr{W}$ when $0<\e<\e_{0}$. The transformation $\mathscr{W}$ is a conjugation from $\mathscr{F}$ to $\hat{\mathscr{F}}$, and $\hat{\mathscr{F}}(\theta,r)=(\theta+\omega(r_{*}),r-\tilde{r})$ is the integrable rotation on $\mathbb{T}\times E$ with frequency $\omega(r_{*})=p$ for $p\in \omega(E^\circ)^\circ$ fixed, where $\tilde{r}$ is the translation about $r$ resulting from the transformation $\mathscr{W}$, and the constant $\tilde{r}\rightarrow 0$ as $\e\rightarrow 0$. That is, the following holds:
 \begin{equation*}
 \mathscr{W}\circ\hat{\mathscr{F}}=\mathscr{F}\circ\mathscr{W} .
 \end{equation*}
 \end{corollary}

Besides concentrating on twist mappings with action-angular variables, Herman \cite{He,H1} first considered the smooth mappings that contain only angular variables. It inspires us to investigate  the perturbed mappings on $ \mathbb{T}^n $ as well. We therefore consider the following mapping $\mathscr{F}:\mathbb{T}^{n}\times \Lambda\rightarrow \mathbb{T}^{n}$ defined by
\begin{equation}\label{equation-3.1}
\theta^{1}=\theta+\omega(\xi)+\e f(\theta,\xi,\e),
\end{equation}
where $\theta\in\mathbb{T}^{n}=\mathbb{R}^{n}/ 2\pi\mathbb{Z}^{n}$, $\xi\in\Lambda\subset\mathbb{R}^{n}$ is a parameter, $\Lambda$ is a connected closed bounded domain with interior points,  and $ \varepsilon $ is a sufficiently small scalar. Assume that the perturbation $f(\theta,\xi,\e)$ is real analytic about $\theta$, continuous about the parameter $\xi$, and the frequency $\omega(\xi)$ is continuous about $\xi$. We will prove that mapping \eqref{equation-3.1} has an invariant torus with the frequency unchanged during the iteration process. Moreover, the assumptions ${\rm (B1)}$-${\rm (B3)}$ corresponding to ${\rm (A1)}$-${\rm (A3)}$ are respectively
\begin{itemize}
\item[{\rm (B1)}] Let $q\in \mathbb{R}^{n}$ be given in advance and denote by $\Lambda^{\circ}$ the interior of the parameter set $\Lambda$. Assume that
\begin{equation}\label{eq-3.3}
\deg\left(\omega(\cdot),\Lambda^{\circ},q\right)\neq 0.
\end{equation}
\item[{\rm (B2)}] Assume that $\omega(\xi_{*})=q$ with some $\xi_*\in \Lambda^{\circ}$ by \eqref{eq-3.3}, and
\begin{equation*}
| \langle k,\omega(\xi_{*})\rangle-k_{0}|\geq \frac{\gamma}{| k|^{\tau}}, \qquad \forall k\in \mathbb{Z}^{n}\backslash\{0\},\ k_{0}\in\mathbb{Z}, \  |k_{0}|\leq M_{1}| k|,
\end{equation*}
where $\tau>n-1$, $\gamma>0$ and $M_{1}$ is assumed to be the upper bound of $|\omega|$ on $\Lambda$.\\
\item[{\rm (B3)}] Assume $B(\xi_{*},\delta)\subset \Lambda^{\circ}$ with $\delta>0$ is the neighborhood of $\xi_{*}
$. There exists a modulus of continuity $\varpi_{2}$ with $\varpi_{1}\leq \varpi_{2}$ such that
\begin{equation*}
| \omega(\xi')-\omega(\xi'')|\geq \varpi_{2}(| \xi'-\xi''|),\qquad  \xi',\xi''\in B(\xi_{*},\delta),\qquad 0<|\xi'-\xi''|\leq1.
\end{equation*}
\end{itemize}

Similar to Theorem \ref{theorem-1}, we have the following theorem on   $ \mathbb{T}^n $, {\textit{where the parameter-dependence for the perturbations is shown to be only continuous.}} This result is new, and unexpected, thanks to the parameter  translation technique introduced in \cite{DL,TD} as we forego.
 \begin{theorem}\label{theorem-2}
 Consider mapping \eqref{equation-3.1}. Assume that the perturbation $f(\theta,\xi,\e)$ is real analytic about $\theta$ on $D(h)$, $\varpi_{1}$ continuous about $\xi$ on $\Lambda$, and $\omega$ is  continuous about $\xi$ on $\Lambda$. Moreover, ${\rm (B1)}$-${\rm (B3)}$ hold. Then there exists a sufficiently small $\e_{0}$, a transformation $\mathscr{U}$ when $0<\e<\e_{0}$. The transformation $ \mathscr{U} $ is a conjugation from $ \mathscr{F} $ to $ \hat{\mathscr{F}} $, and $\hat{\mathscr{F}}(\theta)=\theta+\omega(\xi_{*})$ is the integrable rotation on $\mathbb{T}^{n}\times \Lambda$ with frequency $\omega(\xi_{*})=q$. That is, the following holds:
 \begin{equation}\label{UUUU}
 \mathscr{U}\circ\hat{\mathscr{F}}=\mathscr{F}\circ\mathscr{U}.
 \end{equation}
 \end{theorem}

The main difference between Theorems $\ref{theorem-1}$ and $\ref{theorem-2}$ is that the analyticity of the perturbation $f$ about $r$ can be used in Theorem $\ref{theorem-1}$ to ensure that $f$ is at least Lipschitz continuous about $r$, that is, there exists  a modulus of continuity $\varpi_{1}(x)=x$. In fact, it prohibits us from extending $r$ to complex strips in the KAM scheme if $f$ is assumed to be only continuous about $r$. However, for Theorem $\ref{theorem-2}$ we consider the case where there is no action variable $r$ but only parameter $\xi$. In this situation, the perturbation $f$ being continuous about $\xi$ is enough, and we will employ condition $({\rm B3})$ directly in the proof of frequency-preserving. Explicitly, the parameter-dependence for $ f $ could be very weak, such as the  arbitrary $ \alpha $-H\"{o}lder continuity $\varpi_{\rm H}^{\alpha}(x)=x^{\alpha}$ with any  $0<\alpha<1$, then $ \varpi_{2} $ in {\rm (B3)} being  the Logarithmic Lipschitz type $\varpi_{\rm {LL}}(x)\sim (-\log x)^{-1}$ as $ x\to 0^+ $ allows for Theorem \ref{theorem-2} due to Remark \ref{remark-1}. Actually, in view of Theorem \ref{nonowh}, we could deal with the case which  admits extremely weak regularity, at least nowhere differentiable. In order to show the wide applicability of Theorem \ref{theorem-2}, we directly give the following  corollary. 

\begin{corollary}\label{COROMH}
	Consider mapping \eqref{equation-3.1}, where the perturbation $f(\theta,\xi,\e)$ is assumed to be real analytic about $\theta$ on $D(h)$ and continuous about $\xi$ on $\Lambda$, but nowhere H\"older  continuous, the frequency mapping $\omega(\xi)$ is  continuous about $\xi$ on $\Lambda$. Besides, assume that ${\rm (B1)}$-${\rm (B3)}$ hold with certain $ \varpi_2 $ weaker than the modulus of continuity $ \varpi_1 $ automatically admitted by $ f $ with respect to $ \xi $. Then the conjugacy \eqref{UUUU} in Theorem \ref{theorem-2} holds as long as $ \varepsilon>0 $ is sufficiently small.
\end{corollary}
\begin{remark}
	One could construct explicit applications  following Example 7.5 in \cite{TD} and we omit here for simplicity.
\end{remark}

\subsection{Further discussions}
Here we make some further discussions, including how to touch the  parameter without dimension limitation problem under our approach, as well as the importance of the weak convexity in preserving prescribed frequency.

\subsubsection{Parameter without dimension limitation}
The parameter without dimension limitation problem, as is known to all, is fundamental and difficult in KAM theory, especially using the classical digging frequency method. More precisely, both the angular variable and the action variable have dimensions of $n$, but the dimension of the parameter may be less than $n$. We will touch this question by employing our parameter translation technique. To this end, let us start with a discussion of the topological conditions {\rm (A1)} and {\rm (B1)}.

As can be seen in the proof, these conditions are proposed to ensure that the new parameters $\hat r_{\nu+1}$ and $\xi_{\nu+1}$ could be found in the next KAM step, while the prescribed frequencies remain unchanged due to frequency equations \eqref{equation-4.122} and  \eqref{eq-5.7}, see  \eqref{DISA1} and \eqref{DISA11} respectively. Here we have used the fact that the non-zero Brouwer degree does not change under small perturbations from the KAM iteration, and therefore the solvability of the frequency equations (\eqref{equation-4.122} and  \eqref{eq-5.7}) remains. Actually, the continuity of the frequency mapping $\omega(\xi)$ with respect to parameter $\xi$ is enough to guarantee this, see the new \textit{range  conditions} that can replace the topological conditions {\rm (A1)} and {\rm (B1)} below:

\begin{itemize}
	\item[{\rm (A1*)}] Let $ p=\omega(r_*)\in {\tilde \Omega } \subset \mathbb{R}^n $ satisfy the Diophantine condition in {\rm (A2)}, where $ {\tilde \Omega } $ is an open set of $ \omega(\Omega) $, and $\Omega\subset E \subset \mathbb{R}^n$ is open.
	\item[{\rm (B1*)}]  Let $ q=\omega(\xi_*)\in {\tilde \Omega } \subset \mathbb{R}^n$ satisfy the Diophantine condition in {\rm (B2)}, where $ {\tilde \Omega } $ is an open set of $ \omega(\Omega) $, and $\Omega\subset \Lambda \subset \mathbb{R}^m$ is open. Here $1 \leq m \leq +\infty$ could be different from $n$.
\end{itemize}
One notices that $ {\omega ^{ - 1}}( {\tilde \Omega } ) $ is also an open set due to the continuity of $ \omega $.  As a result, as long as the perturbations in KAM  are sufficiently small,  the solvability of the frequency equations \eqref{equation-4.122} and  \eqref{eq-5.7} do not change thanks to the continuity of $ \omega $ (note that we avoid the boundary of range), and the uniform convergence of $\{r_\nu\}$ and $\{\xi_\nu\}$ could still be obtained by Cauchy theorem through weak convexity  conditions {\rm (A3)} and {\rm (B3)}. Besides, the Brouwer degree requires that the domain of definition and range of mapping should be of the same dimension, while the range condition {\rm (A1*)} and {\rm (B1*)} removes this limitation. Consequently, we directly give the following conclusion.

\begin{theorem}
	Replace {\rm (A1)} and {\rm (B1)} with {\rm (A1*)} and {\rm (B1*)} respectively, leaving the other assumptions unchanged. Then the frequency-preserving KAM persistence in Theorem \ref{theorem-1}, Corollary \ref{cor-1}, Theorem \ref{theorem-2} and Corollary \ref{COROMH} is still allowed. Especially, for Theorem  \ref{theorem-2} and Corollary \ref{COROMH} related to  perturbed mapping with parameter, the dimension of  parameter could be different from that for angular variable.
\end{theorem}

\subsubsection{Weak convexity}
We end this section by making some comments on our weak convexity  conditions {\rm (A3)} and {\rm (B3)}. Such conditions were firstly proposed in \cite{DL,TD} to keep the prescribed frequency in Hamiltonian systems unchanged, \textit{and were shown to be unremovable in the sense of frequency-preserving, see the counterexample constructed in \cite{DL}.} Although the KAM theorems of the mapping form are somewhat  different from the former, the weak convexity conditions still ensure  frequency-preserving KAM persistence, as shown in Theorems \ref{theorem-1} and \ref{theorem-2}.

\section{KAM steps}\label{SEC4}
 \setcounter{equation}{0}
 In this section, we will show details of one cycle of KAM steps. Throughout this paper, $c$ is used to denote an intermediate positive constant, and $c_{1}-c_{4}$ are positive constants. All of them are independent of the iteration process.
\subsection{Description of the $0$-th KAM step}\label{section-2}
For sufficiently large integer $m$, let $\rho$ be a constant with $0<\rho<1$, and assume $\eta>0$ such that $(1+\rho)^{\eta}>2$. Define
\begin{equation*}
\gamma=\e^{\frac{1}{4(n+m+2)}}.
\end{equation*}
The parameters in the $0$-th KAM step are defined by
\begin{equation*}
h_{0}=h,\qquad s_{0}= s,\qquad \gamma_{0}= \gamma,\qquad \mu_{0}= \e^{\frac{1}{8\eta(m+1)}},
\end{equation*}
\begin{equation*}
D(h_{0})=\{\theta\in \mathbb{C}^{n}: {\rm Re}\ \theta\in\mathbb{T}^{n}, \ |{\rm Im}\ \theta|\leq h_{0}\},\qquad G(s_{0})=\{r\in \mathbb{C}^{n}:{\rm Re}\ r\in E, \ |{\rm Im}\ r|\leq s_{0}\},
\end{equation*}
where $0<s_{0},h_{0},\gamma_{0},\mu_{0}\leq 1$, and denote $\mathcal{D}_{0}:=\mathcal{D}(h_{0},s_{0})=D(h_{0})\times G(s_{0})$ for simplicity.

The mapping at $0$-th KAM step is
\begin{equation*}\mathscr{F}_{0}:
\begin{cases}
\theta^{1}_{0}=\theta_{0}+\omega_{0}(r_{0})+f_{0}(\theta_{0},r_{0},\e),\\
r^{1}_{0}=r_{0}+g_{0}(\theta_{0},r_{0},\e),
\end{cases}
\end{equation*}
where $\omega_{0}(r_{0})=\omega(r_{*})=p$, $f_{0}(\theta_{0},r_{0},\e)=\e f(\theta_{0},r_{0},\e)$, and $ g_{0}(\theta_{0},r_{0},\e)=\e g(\theta_{0},r_{0},\e)$.
The following lemma states the estimates on perturbations $f_{0}$ and $g_{0}$.
\begin{lemma}Assume $\e_{0}$ is sufficiently small so that
\begin{equation*}
\e^{\frac{3}{4}}(||f||_{\mathcal{D}_{0}}+||g||_{\mathcal{D}_{0}})\leq s^{m}_{0}\e^{\frac{1}{8\eta(m+1)}},
\end{equation*}
for $0<\e<\e_{0}$.
Then
\begin{equation*}
||f_{0}||_{\mathcal{D}_{0}}+  ||g_{0}||_{\mathcal{D}_{0}}\leq \gamma^{n+m+2}_{0}s^{m}_{0}\mu_{0}.
\end{equation*}
\end{lemma}
\begin{proof}Following $\gamma_{0}^{n+m+2}=\e^{\frac{1}{4}}$ and $\mu_{0}=\e^{\frac{1}{8\eta(m+1)}}$, one has
\begin{align*}
\gamma^{n+m+2}_{0}s^{m}_{0}\mu_{0}&= s^{m}_{0}\e^{\frac{1}{4}}\e^{\frac{1}{8\eta(m+1)}}\\
&\geq s^{m}_{0}\e^{\frac{1}{4}}\e^{\frac{1}{8\eta(m+1)}}s^{-m}_{0}\e^{\frac{3}{4}}\e^{-\frac{1}{8\eta(m+1)}}(||f||_{\mathcal{D}_{0}}+||g||_{\mathcal{D}_{0}})\\
&=\e(||f||_{\mathcal{D}_{0}}+||g||_{\mathcal{D}_{0}})\\
&=||f_{0}||_{\mathcal{D}_{0}}+  ||g_{0}||_{\mathcal{D}_{0}}.
\end{align*}
\end{proof}
\subsection{Description of the $\nu$-th KAM step }
We now define the parameters appear in $\nu$-th KAM step:
\begin{equation*}
h_{\nu}=\frac{h_{\nu-1}}{2}+\frac{h_{0}}{4},\quad s_{\nu}=\frac{s_{\nu-1}}{2},\quad \mu_{\nu}=\mu^{1+\rho}_{\nu-1},\quad \mathcal{D}_{\nu}=\mathcal{D}(h_{\nu},s_{\nu}).
\end{equation*}
After $\nu$ KAM steps, the mapping becomes
\begin{equation*}\mathscr{F}_{\nu}:
\begin{cases}
\theta^{1}_{\nu}=\theta_{\nu}+\omega_{0}(r_{0})+f_{\nu}(\theta_{\nu},r_{\nu},\e),\\
r^{1}_{\nu}=r_{\nu}-\sum\limits^{\nu}_{i=0}r^{*}_{i}+g_{\nu}(\theta_{\nu},r_{\nu},\e).
\end{cases}
\end{equation*}
Moreover,
\begin{equation}\label{equation-4.4}
||f_{\nu}||_{\mathcal{D}_{\nu}}+  ||g_{\nu}||_{\mathcal{D}_{\nu}}\leq \gamma^{n+m+2}_{0}s^{m}_{\nu}\mu_{\nu}.
\end{equation}
Define
\begin{align*}
h_{\nu+1}&=\frac{h_{\nu}}{2}+\frac{h_{0}}{4},\\
s_{\nu+1}&=\frac{s_{\nu}}{2},\\
\mu_{\nu+1}&=\mu^{1+\rho}_{\nu},\\
K_{\nu+1}&=([\log\frac{1}{\mu_{\nu}}]+1)^{3\eta},\\
\mathscr{D}_{i}&=\mathcal{D}(h_{\nu+1}+\frac{i-1}{4}(h_{\nu}-h_{\nu+1}),is_{\nu+1}),\quad i=1,2,3,4,\\
\mathcal{D}_{\nu+1}&=\mathcal{D}(h_{\nu+1},s_{\nu+1}),\\
\hat{\mathcal{D}}_{\nu+1}&=\mathcal{D}(h_{\nu+2}+\frac{3}{4}(h_{\nu+1}-h_{\nu+2}),s_{\nu+2}),\\
\Gamma(h_{\nu}-h_{\nu+1})&=\sum_{0<|k|\leq K_{\nu+1}}|k|^{\tau}e^{-|k|\frac{h_{\nu}-h_{\nu+1}}{4}}\leq \frac{4^{\tau}\tau!}{(h_{\nu}-h_{\nu+1})^{\tau}}.
\end{align*}
For simplicity of notation, we also denote
\begin{align*}
&\mathscr{D}_{3}:=\mathscr{D}_{*}\times\mathscr{G}_{*}:=D(h_{\nu+1}+\frac{1}{2}(h_{\nu}-h_{\nu+1}))\times G(3s_{\nu+1}),\\
&\mathscr{D}_{4}:=\mathscr{D}_{**}\times\mathscr{G}_{**}:=D(h_{\nu+1}+\frac{3}{4}(h_{\nu}-h_{\nu+1}))\times G(4s_{\nu+1}).\\
\end{align*}
\subsubsection{Truncation}
The Fourier series expansion of $f_{\nu}(\theta_{\nu+1},r_{\nu+1},\e)$ is
\begin{equation}\label{equation-42}
f_{\nu}(\theta_{\nu+1},r_{\nu+1},\e)=\sum_{k\in\mathbb{Z}^{n}}f_{k,\nu}(r_{\nu+1})e^{{\rm i}\langle k,\theta_{\nu+1}\rangle},
\end{equation}
where $f_{k,\nu}(r_{\nu+1})=\int_{\mathbb{T}^n}f_{\nu}(\theta_{\nu+1},r_{\nu+1},\e)e^{-{\rm i}\langle k,\theta_{\nu+1}\rangle}\,d\theta_{\nu+1}$ is the Fourier coefficient of $f_{\nu}$.
The truncation and remainder of $f_{\nu}(\theta_{\nu+1},r_{\nu+1},\e)$ are respectively
\begin{align*}
\mathcal{T}_{K_{\nu+1}}f_{\nu}(\theta_{\nu+1},r_{\nu+1},\e)&=\sum_{0<|k|\leq K_{\nu+1}}f_{k,\nu}e^{{\rm i}\langle k,\theta_{\nu+1}\rangle},\\
\mathcal{R}_{K_{\nu+1}}f_{\nu}(\theta_{\nu+1},r_{\nu+1},\e)&=\sum_{|k|> K_{\nu+1}}f_{k,\nu}e^{{\rm i}\langle k,\theta_{\nu+1}\rangle}.
\end{align*}
Thus, $f_{\nu}$ has an equivalent expression of the form
\begin{equation*}
f_{\nu}(\theta_{\nu+1},r_{\nu+1},\e)=f_{0,\nu}(r_{\nu+1})+\mathcal{T}_{K_{\nu+1}}f_{\nu}(\theta_{\nu+1},r_{\nu+1},\e)+\mathcal{R}_{K_{\nu+1}}f_{\nu}(\theta_{\nu+1},r_{\nu+1},\e).
\end{equation*}
Furthermore, we have the following estimate.
\begin{lemma}
If
\begin{equation*}
{(\rm H1)}\qquad\int^{+\infty}_{K_{\nu+1}}l^{n}e^{-l\frac{h_{\nu}-h_{\nu+1}}{4}}\,dl\leq \mu_{\nu},
\end{equation*}
then there exists a constant $c_{1}$ such that
\begin{equation*}
|| \mathcal{R}_{K_{\nu+1}}f_{\nu}||_{\mathscr{D}_{3}}\leq c_{1}\gamma^{n+m+2}_{0}s^{m}_{\nu}\mu^{2}_{\nu}.
\end{equation*}
\end{lemma}
\begin{proof}
 Since the Fourier coefficients decay exponentially, one has
\begin{align*}
| f_{k,\nu}|_{\mathscr{G}_{*}}
&\leq | f_{\nu}| _{\mathscr{D}_{4}}e^{-| k| (h_{\nu+1}+\frac{3}{4}(h_{\nu}-h_{\nu+1}))}\\
&\leq \gamma^{n+m+2}_{0}s^{m}_{\nu}\mu_{\nu}e^{-| k| (h_{\nu+1}+\frac{3}{4}(h_{\nu}-h_{\nu+1}))},
\end{align*}
then
\begin{align*}
| \mathcal{R}_{K_{\nu+1}}f_{\nu}|_{\mathscr{D}_{3}}&\leq\sum_{| k|>K_{\nu+1}}| f_{k,\nu}|_{\mathscr{G}_{*}} e^{| k|(h_{\nu+1}+\frac{1}{2}(h_{\nu}-h_{\nu+1}))}\\
&\leq\sum_{| k|> K_{\nu+1}}| f_{\nu}| _{\mathscr{D}_{4}}e^{-| k| \frac{h_{\nu}-h_{\nu+1}}{4}}\\
&\leq\gamma^{n+m+2}_{0}s^{m}_{\nu}\mu_{\nu}\sum_{| k|>K_{\nu+1}}e^{-| k|\frac{h_{\nu}-h_{\nu+1}}{4}}\\
&\leq\gamma^{n+m+2}_{0}s^{m}_{\nu}\mu_{\nu}\int^{+\infty}_{K_{\nu+1}}l^{n}e^{-l\frac{h_{\nu}-h_{\nu+1}}{4}}\,dl\\
&\leq\gamma^{n+m+2}_{0}s^{m}_{\nu}\mu^{2}_{\nu}.
\end{align*}
Moreover,
\begin{align*}
[ \mathcal{R}_{K_{\nu+1}}f_{\nu}]_{\varpi_{1}}&=\sup_{\theta_{\nu+1}\in \mathscr{D}_{*} }\sup_{\substack{ r'_{\nu+1},r''_{\nu+1}\in \mathscr{G}_{*}\\ r'_{\nu+1}\neq r''_{\nu+1}}}\frac{| \mathcal{R}_{K_{\nu+1}}f_{\nu}(\theta_{\nu+1},r'_{\nu+1},\e)-\mathcal{R}_{K_{\nu+1}}f_{\nu}(\theta_{\nu+1},r''_{\nu+1},\e)|}{\varpi_{1}(| r'_{\nu+1}-r''_{\nu+1}|)}\\
&=\sup_{\theta_{\nu+1}\in \mathscr{D}_{*} }\sup_{\substack{ r'_{\nu+1},r''_{\nu+1}\in \mathscr{G}_{*}\\ r'_{\nu+1}\neq r''_{\nu+1}}}\frac{\big|\sum\limits_{| k|>K_{\nu+1}} f_{k,\nu}(r'_{\nu+1})e^{{\rm i}\langle k,\theta_{\nu+1}\rangle}-\sum\limits_{| k|>K_{\nu+1}}f_{k,\nu}(r''_{\nu+1})e^{{\rm i}\langle k,\theta_{\nu+1}\rangle}\big|}{\varpi_{1}(| r'_{\nu+1}-r''_{\nu+1}|)}\\
&\leq\sup_{\theta_{\nu+1}\in \mathscr{D}_{*} }\sup_{\substack{ r'_{\nu+1},r''_{\nu+1}\in \mathscr{G}_{*}\\ r'_{\nu+1}\neq r''_{\nu+1}}}\frac{\sum\limits_{| k|>K_{\nu+1}}| f_{k,\nu}(r'_{\nu+1})-f_{k,\nu}(r''_{\nu+1})| e^{|k|(h_{\nu+1}+\frac{1}{2}(h_{\nu}-h_{\nu+1}))}}{\varpi_{1}(| r'_{\nu+1}-r''_{\nu+1}|)}\\
&\leq\sup_{\theta_{\nu+1}\in \mathscr{D}_{**} }\sup_{\substack{ r'_{\nu+1},r''_{\nu+1}\in \mathscr{G}_{**}\\ r'_{\nu+1}\neq r''_{\nu+1}}}\frac{| f_{\nu}(\theta_{\nu+1},r'_{\nu+1},\e)-f_{\nu}(\theta_{\nu+1},r''_{\nu+1},\e)| \sum\limits_{|k|>K_{\nu+1}} e^{-|k|\frac{h_{\nu}-h_{\nu+1}}{4}}}{\varpi_{1}(|r'_{\nu+1}-r''_{\nu+1}|)}\\
&\leq[f_{\nu}]_{\varpi_{1}}\sum_{| k|>K_{\nu+1}}e^{-|k|\frac{h_{\nu}-h_{\nu+1}}{4}}\\
&\leq[f_{\nu}]_{\varpi_{1}}\sum_{| k|>K_{\nu+1}}k^{n}e^{-|k|\frac{h_{\nu}-h_{\nu+1}}{4}}\\
&\leq[f_{\nu}]_{\varpi_{1}}\int^{+\infty}_{K_{\nu+1}}l^{n}e^{-l\frac{h_{\nu}-h_{\nu+1}}{4}}\,dl\\
&\leq\gamma^{n+m+2}_{0}s^{m}_{\nu}\mu^{2}_{\nu},
\end{align*}
i.e.,
\begin{equation*}
||\mathcal{R}_{K_{\nu+1}}f_{\nu}||_{\mathscr{D}_{3}}=|\mathcal{R}_{K_{\nu+1}}f_{\nu}|_{\mathscr{D}_{3}}+[\mathcal{R}_{K_{\nu+1}}f_{\nu}]_{\varpi_{1}}\leq c_{1}\gamma^{n+m+2}_{0}s^{m}_{\nu}\mu^{2}_{\nu}.
\end{equation*}
\end{proof}
Similarly, we get
\begin{equation*}
||\mathcal{R}_{K_{\nu+1}}g_{\nu}||_{\mathscr{D}_{3}}\leq c_{1}\gamma^{n+m+2}_{0}s^{m}_{\nu}\mu^{2}_{\nu}.
\end{equation*}
\subsubsection{Transformation}
For $\mathscr{F}_{\nu}$, on $\mathcal{D}_{\nu+1}$, introduce a transformation $\mathscr{U}_{\nu+1}:={\rm id}+(U_{\nu+1},V_{\nu+1})$ that satisfies
\begin{equation}\label{equation-4.2}
\mathscr{U}_{\nu+1}\circ\bar{\mathscr{F}}_{\nu+1}=\mathscr{F}_{\nu}\circ\mathscr{U}_{\nu+1},
\end{equation}
where ${\rm id}$ denotes the identity mapping. Since
\begin{equation*}
\mathscr{U}_{\nu+1}:
\begin{cases}
\theta^{1}_{\nu}=\theta^{1}_{\nu+1}+U_{\nu+1}(\theta^{1}_{\nu+1},r^{1}_{\nu+1}),\\
r^{1}_{\nu}=r^{1}_{\nu+1}+V_{\nu+1}(\theta^{1}_{\nu+1},r^{1}_{\nu+1}),
\end{cases}
\end{equation*}
and
\begin{equation*}\bar{\mathscr{F}}_{\nu+1}:
\begin{cases}
\theta^{1}_{\nu+1}=\theta_{\nu+1}+\omega_{0}(r_{0})+\bar{f}_{\nu+1}(\theta_{\nu+1},r_{\nu+1},\e),\\
r^{1}_{\nu+1}=r_{\nu+1}-\sum\limits^{\nu}_{i=0}r^{*}_{i}+\bar{g}_{\nu+1}(\theta_{\nu+1},r_{\nu+1},\e),
\end{cases}
\end{equation*}
with $r^{*}_{0}=0$, from the left side of \eqref{equation-4.2}, we can derive that
\begin{align*}
\theta^{1}_{\nu}=&\theta^{1}_{\nu+1}+U_{\nu+1}(\theta^{1}_{\nu+1},r^{1}_{\nu+1})\\
=&\theta_{\nu+1}+\omega_{0}(r_{0})+\bar{f}_{\nu+1}(\theta_{\nu+1},r_{\nu+1},\e)\\
+&U_{\nu+1}(\theta_{\nu+1}+\omega_{0}(r_{0})+\bar{f}_{\nu+1},r_{\nu+1}-\sum\limits^{\nu}_{i=0}r^{*}_{i}+\bar{g}_{\nu+1}),\\
r^{1}_{\nu}=&r^{1}_{\nu+1}+V_{\nu+1}(\theta^{1}_{\nu+1},r^{1}_{\nu+1})\\
=&r_{\nu+1}-\sum\limits^{\nu}_{i=0}r^{*}_{i}+\bar{g}_{\nu+1}(\theta_{\nu+1},r_{\nu+1},\e)\\
+&V_{\nu+1}(\theta_{\nu+1}+\omega_{0}(r_{0})+\bar{f}_{\nu+1},r_{\nu+1}-\sum\limits^{\nu}_{i=0}r^{*}_{i}+\bar{g}_{\nu+1}).
\end{align*}
Also, one has
\begin{equation*}\mathscr{F}_{\nu}:
\begin{cases}
\theta^{1}_{\nu}=\theta_{\nu}+\omega_{0}(r_{0})+f_{\nu}(\theta_{\nu},r_{\nu},\e),\\
r^{1}_{\nu}=r_{\nu}-\sum\limits^{\nu}_{i=0}r^{*}_{i}+g_{\nu}(\theta_{\nu},r_{\nu},\e),
\end{cases}
\end{equation*}
and
\begin{equation*}\mathscr{U}_{\nu+1}:
\begin{cases}
\theta_{\nu}=\theta_{\nu+1}+U_{\nu+1}(\theta_{\nu+1},r_{\nu+1}-\sum\limits^{\nu}_{i=0}r^{*}_{i}),\\
r_{\nu}=r_{\nu+1}+V_{\nu+1}(\theta_{\nu+1},r_{\nu+1}-\sum\limits^{\nu}_{i=0}r^{*}_{i}).
\end{cases}
\end{equation*}
By the right side of \eqref{equation-4.2}, we obtain
\begin{align*}
\theta^{1}_{\nu}=&\theta_{\nu}+\omega_{0}(r_{0})+f_{\nu}(\theta_{\nu},r_{\nu},\e)\\
=&\theta_{\nu+1}+U_{\nu+1}(\theta_{\nu+1},r_{\nu+1}-\sum\limits^{\nu}_{i=0}r^{*}_{i})+\omega_{0}(r_{0})\\
+&f_{\nu}(\theta_{\nu+1}+U_{\nu+1},r_{\nu+1}+V_{\nu+1},\e),\\
r^{1}_{\nu}=&r_{\nu}-\sum\limits^{\nu}_{i=0}r^{*}_{i}+g_{\nu}(\theta_{\nu},r_{\nu},\e)\\
=&r_{\nu+1}+V_{\nu+1}(\theta_{\nu+1},r_{\nu+1}-\sum\limits^{\nu}_{i=0}r^{*}_{i})-\sum\limits^{\nu}_{i=0}r^{*}_{i}\\
+& g_{\nu}(\theta_{\nu+1}+U_{\nu+1},r_{\nu+1}+V_{\nu+1},\e).
\end{align*}
 Therefore,
 \eqref{equation-4.2} implies that
\begin{align}\label{equation-43}
&\omega_{0}(r_{0})+\bar{f}_{\nu+1}(\theta_{\nu+1},r_{\nu+1},\e)+U_{\nu+1}(\theta_{\nu+1}+\omega_{0}(r_{0})+\bar{f}_{\nu+1},r_{\nu+1}-\sum\limits^{\nu}_{i=0}r^{*}_{i}+\bar{g}_{\nu+1})\notag\\
=&U_{\nu+1}(\theta_{\nu+1},r_{\nu+1}-\sum\limits^{\nu}_{i=0}r^{*}_{i})+\omega_{0}(r_{0})+f_{\nu}(\theta_{\nu+1}+U_{\nu+1},r_{\nu+1}+V_{\nu+1},\e),\\
\label{equation-44}
&\bar{g}_{\nu+1}(\theta_{\nu+1},r_{\nu+1},\e)+V_{\nu+1}(\theta_{\nu+1}+\omega_{0}(r_{0})+\bar{f}_{\nu+1},r_{\nu+1}-\sum\limits^{\nu}_{i=0}r^{*}_{i}+\bar{g}_{\nu+1})\notag\\
=&V_{\nu+1}(\theta_{\nu+1},r_{\nu+1}-\sum\limits^{\nu}_{i=0}r^{*}_{i})+g_{\nu}(\theta_{\nu+1}+U_{\nu+1},r_{\nu+1}+V_{\nu+1},\e).
\end{align}
Following the iteration process before, the $\omega_{0}(r_{0})$ on the right side of \eqref{equation-44} is actually
\begin{equation*}
\omega_{0}(r_{0})=\omega_{0}(r_{\nu})+\sum^{\nu-1}_{i=0}f_{0,i}(r_{\nu}).
\end{equation*}
Since the frequency $\omega_{0}$ is continuous about $r$, assume that there exists a modulus of continuity $\varpi_{*}$ such that
\begin{equation}\label{eq-46}
[\omega_{0}]_{\varpi_{*}}=\frac{|\omega_{0}(r_{\nu})-\omega_{0}(r_{\nu+1})|}{\varpi_{*}(|r_{\nu}-r_{\nu+1}|)}<+\infty.
\end{equation}
The perturbation $f$ is $\varpi_{1}$ continuous about $r$ with $\varpi_{1}\leq \varpi_{*}$ due to $\varpi_{1}(x)=x$, one has
\begin{equation}\label{eq-47}
[f_{0,i}]_{\varpi_{1}}=\frac{|f_{0,i}(r_{\nu})-f_{0,i}(r_{\nu+1})|}{\varpi_{1}(|r_{\nu}-r_{\nu+1}|)}<+\infty,\qquad 0\leq i\leq \nu-1 .
\end{equation}
Consequently, from \eqref{eq-46} and \eqref{eq-47}, we may deduce that
\begin{equation*}
\omega_{0}(r_{\nu})=\omega_{0}(r_{\nu+1})+\mathcal{O}(\varpi_{*}(|V_{\nu+1}|)),
\end{equation*}
and
\begin{equation*}
f_{0,i}(r_{\nu})=f_{0,i}(r_{\nu+1})+\mathcal{O}(\varpi_{1}(|V_{\nu+1}|)).
\end{equation*}
Then \eqref{equation-43} and \eqref{equation-44} are equal to the following
\begin{align*}
&U_{\nu+1}(\theta_{\nu+1}+\omega_{0}(r_{0})+\bar{f}_{\nu+1},r_{\nu+1}-\sum\limits^{\nu}_{i=0}r^{*}_{i}+\bar{g}_{\nu+1})-U_{\nu+1}(\theta_{\nu+1}+\omega_{0}(r_{0}),r_{\nu+1}-\sum\limits^{\nu}_{i=0}r^{*}_{i})\notag\\
+&U_{\nu+1}(\theta_{\nu+1}+\omega_{0}(r_{0}),r_{\nu+1}-\sum\limits^{\nu}_{i=0}r^{*}_{i})-U_{\nu+1}(\theta_{\nu+1},r_{\nu+1}-\sum\limits^{\nu}_{i=0}r^{*}_{i})\notag\\
+&\omega_{0}(r_{0})+\bar{f}_{\nu+1}(\theta_{\nu+1},r_{\nu+1},\e)\notag\\
=&\omega_{0}(r_{\nu+1})+\sum^{\nu-1}_{i=0}f_{0,i}(r_{\nu+1})+\mathcal{O}(\varpi_{*}(|V_{\nu+1}|))+\mathcal{O}(\varpi_{1}(|V_{\nu+1}|))\notag\\
+&f_{\nu}(\theta_{\nu+1}+U_{\nu+1},r_{\nu+1}+V_{\nu+1},\e)-f_{\nu}(\theta_{\nu+1},r_{\nu+1},\e)+f_{0,\nu}(r_{\nu+1})+\mathcal{T}_{K_{\nu+1}}f_{\nu}(\theta_{\nu+1},r_{\nu+1},\e)\\
+&\mathcal{R}_{K_{\nu+1}}f_{\nu}(\theta_{\nu+1},r_{\nu+1},\e),\\
&V_{\nu+1}(\theta_{\nu+1}+\omega_{0}(r_{0})+\bar{f}_{\nu+1},r_{\nu+1}-\sum\limits^{\nu}_{i=0}r^{*}_{i}+\bar{g}_{\nu+1})-V_{\nu+1}(\theta_{\nu+1}+\omega_{0}(r_{0}),r_{\nu+1}-\sum\limits^{\nu}_{i=0}r^{*}_{i})\notag\\
+&V_{\nu+1}(\theta_{\nu+1}+\omega_{0}(r_{0}),r_{\nu+1}-\sum\limits^{\nu}_{i=0}r^{*}_{i})-V_{\nu+1}(\theta_{\nu+1},r_{\nu+1}-\sum\limits^{\nu}_{i=0}r^{*}_{i})+\bar{g}_{\nu+1}(\theta_{\nu+1},r_{\nu+1},\e)\notag\\
=&g_{\nu}(\theta_{\nu+1}+U_{\nu+1},r_{\nu+1}+V_{\nu+1},\e)-g_{\nu}(\theta_{\nu+1},r_{\nu+1},\e)+g_{0,\nu}(r_{\nu+1})+\mathcal{T}_{K_{\nu+1}}g_{\nu}(\theta_{\nu+1},r_{\nu+1},\e)\\
+&\mathcal{R}_{K_{\nu+1}}g_{\nu}(\theta_{\nu+1},r_{\nu+1},\e).
\end{align*}
The transformation $\mathscr{U}_{\nu+1}={\rm id}+(U_{\nu+1},V_{\nu+1})$ needs to satisfy the homological equations
\begin{equation}\label{eq-4.8}
\begin{cases}
U_{\nu+1}(\theta_{\nu+1}+\omega_{0}(r_{0}),r_{\nu+1}-\sum\limits^{\nu}_{i=0}r^{*}_{i})-U_{\nu+1}(\theta_{\nu+1},r_{\nu+1}-\sum\limits^{\nu}_{i=0}r^{*}_{i})=\mathcal{T}_{K_{\nu+1}}f_{\nu}(\theta_{\nu+1},r_{\nu+1},\e),\\
V_{\nu+1}(\theta_{\nu+1}+\omega_{0}(r_{0}),r_{\nu+1}-\sum\limits^{\nu}_{i=0}r^{*}_{i})-V_{\nu+1}(\theta_{\nu+1},r_{\nu+1}-\sum\limits^{\nu}_{i=0}r^{*}_{i})=\mathcal{T}_{K_{\nu+1}}g_{\nu}(\theta_{\nu+1},r_{\nu+1},\e).
\end{cases}
\end{equation}
The new perturbations are respectively
\begin{align}\label{eq-4.10}
\bar{f}_{\nu+1}(\theta_{\nu+1},r_{\nu+1},\e)&=f_{\nu}(\theta_{\nu+1}+U_{\nu+1},r_{\nu+1}+V_{\nu+1},\e)-f_{\nu}(\theta_{\nu+1},r_{\nu+1},\e)\notag\\
&+\mathcal{R}_{K_{\nu+1}}f_{\nu}(\theta_{\nu+1},r_{\nu+1},\e)+\mathcal{O}(\varpi_{*}(|V_{\nu+1}|))+\mathcal{O}(\varpi_{1}(|V_{\nu+1}|))\notag\\
&+U_{\nu+1}(\theta_{\nu+1}+\omega_{0}(r_{0}),r_{\nu+1}-\sum^{\nu}_{i=0}r^{*}_{i})\notag\\
&-U_{\nu+1}(\theta_{\nu+1}+\omega_{0}(r_{0})+\bar{f}_{\nu+1},r_{\nu+1}-\sum^{\nu}_{i=0}r^{*}_{i}+\bar{g}_{\nu+1}),\\
\label{eq-4.11}
\bar{g}_{\nu+1}(\theta_{\nu+1},r_{\nu+1},\e)&=g_{\nu}(\theta_{\nu+1}+U_{\nu+1},r_{\nu+1}+V_{\nu+1},\e)-g_{\nu}(\theta_{\nu+1},r_{\nu+1},\e)\notag\\
&+g_{0,\nu}(r_{\nu+1})+\mathcal{R}_{K_{\nu+1}}g_{\nu}(\theta_{\nu+1},r_{\nu+1},\e)\notag\\
&+V_{\nu+1}(\theta_{\nu+1}+\omega_{0}(r_{0}),r_{\nu+1}-\sum^{\nu}_{i=0}r^{*}_{i})\notag\\
&-V_{\nu+1}(\theta_{\nu+1}+\omega_{0}(r_{0})+\bar{f}_{\nu+1},r_{\nu+1}-\sum^{\nu}_{i=0}r^{*}_{i}+\bar{g}_{\nu+1}).
\end{align}

The homological equations \eqref{eq-4.8} 
are uniquely solvable on $\mathcal{D}_{\nu+1}$.  Let us start by considering the first equation in \eqref{eq-4.8}. 
Formally, denote $U_{\nu+1}(\theta_{\nu+1},r_{\nu+1}-\sum\limits^{\nu}_{i=0}r^{*}_{i})$ as
\begin{equation*}
U_{\nu+1}(\theta_{\nu+1},r_{\nu+1}-\sum\limits^{\nu}_{i=0}r^{*}_{i})=\sum_{0<|k|\leq K_{\nu+1}}U_{k,\nu}e^{{\rm i}\langle k,\theta_{\nu+1}\rangle },
\end{equation*}
taking it into the first equation in \eqref{eq-4.8}, one has 
\begin{equation*}
\sum_{0<|k|\leq K_{\nu+1}}U_{k,\nu+1}e^{{\rm i}\langle k,\theta_{\nu+1}+\omega_{0}(r_{0})\rangle}-\sum_{0<|k|\leq K_{\nu+1}}U_{k,\nu+1}e^{{\rm i}\langle k,\theta_{\nu+1}\rangle}=\sum_{0<|k|\leq K_{\nu+1}}f_{k,\nu}e^{{\rm i}\langle k,\theta_{\nu+1}\rangle}.
\end{equation*}
By comparing the coefficients above, we have
\begin{equation}\label{equation-4.12}
U_{k,\nu+1}(e^{{\rm i}\langle k,\omega_{0}(r_{0})\rangle}-1)=f_{k,\nu}.
\end{equation}
The details of estimating \eqref{equation-4.12} can be seen in the following lemma.

\begin{lemma}

The equation \eqref{equation-4.12} has a unique solution $U_{k,\nu+1}$ on $G(s_{\nu+1})$ satisfying the following estimate
\begin{equation*}
||U_{k,\nu+1}||_{\mathscr{G}_{*}}\leq c_{2}||f_{\nu}||_{\mathscr{D}_{4}}\gamma^{-1}_{0}|k|^{\tau}e^{-|k|(h_{\nu+1}+\frac{3}{4}(h_{\nu}-h_{\nu+1}))}.
\end{equation*}
\end{lemma}
\begin{proof}We notice that the coefficients $f_{k,\nu}$ decay exponentially, i.e.,
\begin{equation*}
||f_{k,\nu}||_{\mathscr{G}_{**}}\leq ||f_{\nu}||_{\mathscr{D}_{4}}e^{-|k|(h_{\nu+1}+\frac{3}{4}(h_{\nu}-h_{\nu+1}))}.
\end{equation*}
There exists a $k_{0}\in\mathbb{Z}$ satisfying $|\frac{\langle k,\omega_{0}(r_{0})\rangle-k_{0}}{2}|\leq \frac{\pi}{2}$ such that
\begin{align*}
||e^{{\rm i}\langle k,\omega_{0}(r_{0})\rangle}-1||_{\mathscr{G}_{**}}&\geq2||\sin\frac{\langle k,\omega_{0}(r_{0})\rangle-k_{0}}{2}||_{\mathscr{G}_{**}}\notag\\
&\geq\frac{4}{\pi}||\frac{\langle k,\omega_{0}(r_{0})\rangle-k_{0}}{2}||_{\mathscr{G}_{**}}\notag\\
&\geq c||\langle k,\omega_{0}(r_{0})\rangle-k_{0}||_{\mathscr{G}_{**}}\notag\\
&\geq\frac{c\gamma_{0}}{|k|^{\tau}}.
\end{align*}
Therefore,
\begin{align}\label{equation-46}
||U_{k,\nu+1}||_{\mathscr{G}_{*}}&\leq \frac{||f_{k,\nu}||_{\mathscr{G}_{**}}}{||e^{{\rm i}\langle k,\omega_{0}(r_{0})\rangle}-1||_{\mathscr{G}_{**}}}\notag\\
&\leq c_{2}||f_{\nu}||_{\mathscr{D}_{4}}\gamma^{-1}_{0}|k|^{\tau}e^{-|k|(h_{\nu+1}+\frac{3}{4}(h_{\nu}-h_{\nu+1}))}.
\end{align}
\end{proof}
In the same way, we get
\begin{eqnarray*}
V_{k,\nu+1}(e^{{\rm i}\langle k,\omega_{0}(r_{0})\rangle}-1)=g_{k,\nu}.
\end{eqnarray*}
Thus,
\begin{align}\label{equation-47}
||V_{k,\nu+1}||_{\mathscr{G}_{*}}&\leq\frac{||g_{k,\nu}||_{\mathscr{G}_{**}}}{||e^{{\rm i}\langle k,\omega_{0}(r_{0})\rangle}-1||_{\mathscr{G}_{**}}}\notag\\
&\leq c_{2}||g_{\nu}||_{\mathscr{D}_{4}}\gamma^{-1}_{0}|k|^{\tau}e^{-|k|(h_{\nu+1}+\frac{3}{4}(h_{\nu}-h_{\nu+1}))}.
\end{align}

\subsubsection{Translation}\label{subsec-4.2.3}
 In the usual iteration process, one has to find out a  decreasing series of domains that the Diophantine condition fails. To avoid this, we will construct a translation to keep the frequency unchanged in this section. Consider the translation
\begin{equation*}
\mathscr{V}_{\nu+1}:\theta_{\nu+1}\rightarrow \theta_{\nu+1},\quad r_{\nu+1}\rightarrow r_{\nu+1}+r^{*}_{\nu+1}:=\hat{r}_{\nu+1},
\end{equation*}
where $\hat{r}_{\nu+1}\in B_{c\mu_{\nu}}(\hat{r}_{\nu})$. The action has a shift under the translation $ \mathscr{V}_{\nu+1} $, but the angular variable is unchanged.  To make the frequency-preserving, it requests that
\begin{equation*}
\omega_{0}(\hat{r}_{\nu+1})+\sum^{\nu}_{i=0}f_{0,i}(\hat{r}_{\nu+1})-\omega_{0}(r_{0})+\omega_{0}(r_{0})=\omega_{0}(r_{0}),
\end{equation*}
i.e.,
\begin{equation*}
\omega_{0}(\hat{r}_{\nu+1})+\sum^{\nu}_{i=0}f_{0,i}(\hat{r}_{\nu+1})=\omega_{0}(r_{0}).
\end{equation*}
We will demonstrate the equation in the following subsection. After the translation $\mathscr{V}_{\nu+1}$, the mapping $\bar{\mathscr{F}}_{\nu+1}$ becomes $\mathscr{F}_{\nu+1}=\bar{\mathscr{F}}_{\nu+1}\circ\mathscr{V}_{\nu+1}$, i.e.,
\begin{equation*}
\mathscr{F}_{\nu+1}:
\begin{cases}
\theta^{1}_{\nu+1}=\theta_{\nu+1}+\omega_{0}(r_{0})+f_{\nu+1}(\theta_{\nu+1},\hat{r}_{\nu+1},\e),\\
\hat{r}^{1}_{\nu+1}=\hat{r}_{\nu+1}-\sum\limits^{\nu}_{i=0}r^{*}_{i}+g_{\nu+1}(\theta_{\nu+1},\hat{r}_{\nu+1},\e).
\end{cases}
\end{equation*}
The corresponding new perturbations are
\begin{align}\label{equation-4.25}
f_{\nu+1}(\theta_{\nu+1},\hat{r}_{\nu+1},\e)&=f_{\nu}(\theta_{\nu+1}+U_{\nu+1},\hat{r}_{\nu+1}+V_{\nu+1},\e)-f_{\nu}(\theta_{\nu+1},\hat{r}_{\nu+1},\e)\notag\\
&+\mathcal{R}_{K_{\nu+1}}f_{\nu}(\theta_{\nu+1},\hat{r}_{\nu+1},\e)+\mathcal{O}(\varpi_{*}(|V_{\nu+1}|))+\mathcal{O}(\varpi_{1}(|V_{\nu+1}|))\notag\\
&+U_{\nu+1}(\theta_{\nu+1}+\omega_{0}(r_{0}),\hat{r}_{\nu+1}-\sum^{\nu}_{i=0}r^{*}_{i})\notag\\
&-U_{\nu+1}(\theta_{\nu+1}+\omega_{0}(r_{0})+f_{\nu+1},\hat{r}_{\nu+1}-\sum^{\nu}_{i=0}r^{*}_{i}+g_{\nu+1}),
\end{align}
and
\begin{align}\label{equation-4.26}
g_{\nu+1}(\theta_{\nu+1},\hat{r}_{\nu+1},\e)&=g_{\nu}(\theta_{\nu+1}+U_{\nu+1},\hat{r}_{\nu+1}+V_{\nu+1},\e)-g_{\nu}(\theta_{\nu+1},\hat{r}_{\nu+1},\e)\notag\\
&+g_{0,\nu}(\hat{r}_{\nu+1})+\mathcal{R}_{K_{\nu+1}}g_{\nu}(\theta_{\nu+1},\hat{r}_{\nu+1},\e)\notag\\
&+V_{\nu+1}(\theta_{\nu+1}+\omega_{0}(r_{0}),\hat{r}_{\nu+1}-\sum^{\nu}_{i=0}r^{*}_{i})\notag\\
&-V_{\nu+1}(\theta_{\nu+1}+\omega_{0}(r_{0})+f_{\nu+1},\hat{r}_{\nu+1}-\sum^{\nu}_{i=0}r^{*}_{i}+g_{\nu+1}).
\end{align}
\subsubsection{Frequency-preserving}
In this section, we will show that the frequency is unchanged during the iteration process under the conditions ${\rm (A1)}$ and ${\rm (A3)}$. The topological degree condition $({\rm A1})$ states that we can find a $\hat{r}_{\nu+1}$ such that the frequency remains preserved. Besides, the weak convexity condition $({\rm A3})$ ensures that $\{\hat{r}_{\nu}\}$ is a Cauchy sequence. The following lemma is crucial to our consideration.
\begin{lemma}\label{lemma-4.5}Assume that
\begin{equation*}
{(\rm H2)}\qquad||\sum^{\nu}_{i=0}f_{0,i}||_{G(s_{\nu+1})}\leq c\mu^{\frac{1}{2}}_{0}.
\end{equation*}
Then there exists a $\hat{r}_{\nu+1}\in B_{c\mu_{\nu}}(\hat{r}_{\nu})$ such that
\begin{equation}\label{equation-4.27}
\omega_{0}(\hat{r}_{\nu+1})+\sum^{\nu}_{i=0}f_{0,i}(\hat{r}_{\nu+1})=\omega_{0}(r_{0}).
\end{equation}
\end{lemma}
\begin{proof}
The proof is an induction on $\nu$. Obviously, $\omega_{0}(r_{0})=\omega_{0}(r_{0})$ when $\nu=0$. Now assume that for some $\nu\geq 1$, one has
\begin{equation}\label{equation-4.29}
\omega_{0}(\hat{r}_{j})+\sum^{j-1}_{i=0}f_{0,i}(\hat{r}_{j})=\omega_{0}(r_{0}), \qquad \hat{r}_{j}\in B_{}(\hat{r}_{j-1})\subset B(r_{0},\delta),\  1\leq j\leq \nu.
\end{equation}
We need to find out a $\hat{r}_{\nu+1}$ in the neighborhood of $\hat{r}_{\nu}$ that satisfies
\begin{equation}\label{equation-4.122}
\omega_{0}(\hat{r}_{\nu+1})+\sum^{\nu}_{i=0}f_{0,i}(\hat{r}_{\nu+1})=\omega_{0}(r_{0}).
\end{equation}
Since $\mu_{0}^{\frac{1}{2}}$ is sufficiently small and the condition ${\rm (A1)}$ holds, we have
\begin{equation}\label{DISA1}
\deg\Big(\omega_{0}(\cdot)+\sum^{\nu}_{i=0}f_{0,i}(\cdot),B(r_{0},\delta),\omega_{0}(r_{0})\Big)=\deg\left(\omega_{0}(\cdot),B(r_{0},\delta),\omega_{0}(r_{0})\right)\neq 0,
\end{equation}
where $ \omega_{0}(r_0)=p $, and $ p \in \mathbb{R}^n $ is given in advance. This shows that there exists at least a $\hat{r}_{\nu+1}\in B(r_{0},\delta)$ with some $\delta>0$ such that \eqref{equation-4.27} holds. Remark \ref{remark-1} tells us that there exists a modulus of continuity $\varpi_{1}(x)=x$ such that $f$ is $\varpi_{1}$ continuous about $r$, and following \eqref{equation-4.4}, one has
\begin{equation*}
[f_{0,i}]_{\varpi_{1}}\leq c\mu_{i},\qquad 0\leq i\leq \nu,
\end{equation*}
i.e.,
\begin{equation*}
|f_{0,i}(\hat{r}_{\nu+1})-f_{0,i}(\hat{r}_{\nu})|\leq c\mu_{i}\varpi_{1}(|\hat{r}_{\nu+1}-\hat{r}_{\nu}|),\qquad 0\leq i\leq \nu.
\end{equation*}
Following Definitions \ref{def-1} and \ref{DE2.3}, and together with $({\rm A3})$, one has
\begin{equation*}
\varlimsup_{x\rightarrow 0^{+}}\frac{x}{\varpi_{2}(x)}<+\infty,
\end{equation*}
this means $\varpi_{1}\leq \varpi_{2}$. The equations \eqref{equation-4.29} and \eqref{equation-4.122} imply that
\begin{equation*}
\omega_{0}(\hat{r}_{\nu+1})+\sum^{\nu}_{i=0}f_{0,i}(\hat{r}_{\nu+1})=\omega_{0}(\hat{r}_{\nu})+\sum^{\nu-1}_{i=0}f_{0,i}(\hat{r}_{\nu}).
\end{equation*}
Then
\begin{align*}
|f_{0,\nu}(\hat{r}_{\nu+1})|&=|\omega_{0}(\hat{r}_{\nu})-\omega_{0}(\hat{r}_{\nu+1})+\sum\limits^{\nu-1}_{i=0}(f_{0,i}(\hat{r}_{\nu})-f_{0,i}(\hat{r}_{\nu+1}))|\notag\\
&\geq|\omega_{0}(\hat{r}_{\nu})-\omega_{0}(\hat{r}_{\nu+1})|-\sum\limits^{\nu-1}_{i=0}|f_{0,i}(\hat{r}_{\nu})-f_{0,i}(\hat{r}_{\nu+1})|\notag\\
&\geq\varpi_{2}(|\hat{r}_{\nu}-\hat{r}_{\nu+1}|)-c(\sum\limits^{\nu-1}_{i=0}\mu_{i})\varpi_{1}(|\hat{r}_{\nu}-\hat{r}_{\nu+1}|)\notag\\
&\geq\frac{\varpi_{2}(|\hat{r}_{\nu}-\hat{r}_{\nu+1}|)}{2}.
\end{align*}
The last inequality holds since $\e$ is sufficiently small such that $c(\sum\limits^{\nu-1}_{i=0}\mu_{i})\leq \frac{1}{2}$, and $\varpi_1  \leq \varpi_2$. Therefore,
\begin{equation}\label{equation-4.33}
|\hat{r}_{\nu}-\hat{r}_{\nu+1}|\leq \varpi^{-1}_{2}(2|f_{0,\nu}(\hat{r}_{\nu+1})|)\leq \varpi^{-1}_{2}(2c\mu_{\nu})\leq c\varpi^{-1}_{1}(2c\mu_{\nu})\leq c\mu_{\nu},
\end{equation}
where the last inequality is due to Definition \ref{def-1}, i.e., $\varlimsup\limits_{x\rightarrow 0^{+}}\frac{x}{\varpi_{1}(x)}<+\infty$. This implies that $\{\hat{r}_{\nu}\}$ is a Cauchy sequence and $\hat{r}_{\nu+1}\in B_{c\mu_{\nu}}(\hat{r}_{\nu})$.
\end{proof}
\subsubsection{Estimates on new transformations}
According to \eqref{equation-46} and \eqref{equation-47}, the estimates on transformations are given in the lemma below.
\begin{lemma}\label{lemma-4.6}Assume that there exists a constant $c_{3}$,  and
\begin{equation*}
{(\rm H3)}\qquad c_{3}\gamma^{n+m+1}_{0}s^{m}_{\nu}\mu_{\nu}\Gamma(h_{\nu}-h_{\nu+1})\leq \min\{s_{\nu+1},\frac{h_{\nu}-h_{\nu+1}}{4}\}.
\end{equation*}
Then the followings hold.
\begin{itemize}
\item[$({\rm i})$]$||\mathscr{U}_{\nu+1}-{\rm id}||_{\mathscr{D}_{3}}\leq c_{3}\gamma^{n+m+1}_{0}s^{m}_{\nu}\mu_{\nu}\Gamma(h_{\nu}-h_{\nu+1})$.
\item[$(\rm ii)$] $\mathscr{U}_{\nu+1}:\mathcal{D}_{\nu+1}\rightarrow \mathcal{D}_{\nu}$.
\item[$(\rm iii)$] Set $\mathscr{W}_{\nu+1}:=\mathscr{U}_{\nu+1}\circ\mathscr{V}_{\nu+1}$, one has $\mathscr{W}_{\nu+1}:\mathcal{D}_{\nu+1}\rightarrow \mathcal{D}_{\nu}$, and
\begin{equation*}
||\mathscr{W}_{\nu+1}-{\rm id}||_{\mathcal{D}_{\nu+1}}\leq c_{3}\gamma^{n+m+1}_{0}s^{m}_{\nu}\mu_{\nu}\Gamma(h_{\nu}-h_{\nu+1}).
\end{equation*}
\end{itemize}
\end{lemma}
\begin{proof}
$(\rm i)$ Since $U_{\nu+1}(\theta_{\nu+1},r_{\nu+1}-\sum\limits^{\nu}_{i=0}r^{*}_{i})=\sum\limits_{0<|k|\leq K_{\nu+1}}U_{k,\nu+1}e^{{\rm i}\langle k,\theta_{\nu+1}\rangle}$, we get
\begin{align}\label{equation-4.36}
||U_{\nu+1}||_{\mathscr{D}_{3}}&\leq ||U_{k,\nu+1}||_{\mathscr{G}_{*}}\sum_{0<|k|\leq K_{\nu+1}}e^{|k|(h_{\nu+1}+\frac{1}{2}(h_{\nu}-h_{\nu+1}))}\notag\\
&\leq||f_{\nu}||_{\mathscr{D}_{4}}\gamma^{-1}_{0}\sum_{0<|k|\leq K_{\nu+1}}|k|^{\tau}e^{-|k|\frac{h_{\nu}-h_{\nu+1}}{4}}\notag\\
&\leq c_{3}\gamma^{n+m+1}_{0}s^{m}_{\nu}\mu_{\nu}\Gamma(h_{\nu}-h_{\nu+1}).
\end{align}
Besides,
\begin{align}\label{equation-4.35}
||V_{\nu+1}||_{\mathscr{D}_{3}}&\leq  ||V_{k,\nu+1}||_{\mathscr{G}_{*}}\sum_{0<|k|\leq K_{\nu+1}}e^{|k|(h_{\nu+1}+\frac{1}{2}(h_{\nu}-h_{\nu+1}))}\notag\\
&\leq||g_{\nu}||_{\mathscr{D}_{4}}\gamma^{-1}_{0}\sum_{0<|k|\leq K_{\nu+1}}|k|^{\tau}e^{-|k|\frac{h_{\nu}-h_{\nu+1}}{4}}\notag\\
&\leq  c_{3}\gamma^{n+m+1}_{0}s^{m}_{\nu}\mu_{\nu}\Gamma(h_{\nu}-h_{\nu+1}).
\end{align}
Thus, $(\rm i)$ is due to \eqref{equation-4.4}, \eqref{equation-4.36}, and \eqref{equation-4.35}.

$(\rm ii)$ By $(\theta_{\nu+1},r_{\nu+1})\in\mathscr{D}_{3}$, $(\rm H3)$ implies that
\begin{align*}
|\theta_{\nu}-\theta_{\nu+1}|&=||U_{\nu+1}||_{\mathscr{D}_{3}}\notag\\
&\leq c_{3}\gamma^{n+m+1}_{0}s^{m}_{\nu}\mu_{\nu}\Gamma(h_{\nu}-h_{\nu+1})\notag\\
&\leq\frac{h_{\nu}-h_{\nu+1}}{4},\\
|r_{\nu}-r_{\nu+1}|&=||V_{\nu+1}||_{\mathscr{D}_{3}}\notag\\
&\leq c_{3}\gamma^{n+m+1}_{0}s^{m}_{\nu}\mu_{\nu}\Gamma(h_{\nu}-h_{\nu+1})\notag\\
&\leq s_{\nu+1}.
\end{align*}
Thus, $\mathscr{U}_{\nu+1}:\mathcal{D}_{\nu+1}\subset\mathscr{D}_{4}\rightarrow \mathscr{D}_{3}\subset\mathcal{D}_{\nu}$.

$(\rm iii)$ now follows from $(\rm i)$ and $(\rm ii)$ immediately.
\end{proof}
\subsubsection{Estimates on the new perturbations}
In what follows, we are able to show the estimates on the new perturbations.
\begin{lemma}\label{lemma-4.7}Assume that
\begin{equation*}
{(\rm H4) } \qquad\gamma^{n+m+1}_{0}s^{m}_{\nu}\mu_{\nu}\Gamma(h_{\nu}-h_{\nu+1})\leq \varpi^{-1}_{*}(\gamma^{n+m+1}_{0}s^{m}_{\nu}\mu^{2}_{\nu}).
\end{equation*}
Then there exists a constant $c_{4}$ such that
\begin{align*}
&||\bar{f}_{\nu+1}||_{\mathcal{D}_{\nu+1}}+||\bar{g}_{\nu+1}||_{\mathcal{D}_{\nu+1}}\\
\leq& c_{4}\gamma^{n+m+2}_{0}s^{m}_{\nu}\mu^{2}_{\nu}(\gamma^{n+m+1}_{0}s^{m}_{\nu}(h_{\nu+1}-h_{\nu+2})^{-1}\Gamma(h_{\nu}-h_{\nu+1})+\gamma^{n+m+1}_{0}s^{m-1}_{\nu}\Gamma(h_{\nu}-h_{\nu+1})+1).
\end{align*}
Moreover, if
\begin{equation*}
{(\rm H5)}\qquad 2^{m}c_{4}\mu^{1-\rho}_{\nu}(\gamma^{n+m+1}_{0}s^{m}_{\nu}(h_{\nu+1}-h_{\nu+2})^{-1}\Gamma(h_{\nu}-h_{\nu+1})+\gamma^{n+m+1}_{0}s^{m-1}_{\nu}\Gamma(h_{\nu}-h_{\nu+1})+1)\leq 1,
\end{equation*}
then
\begin{equation*}
||f_{\nu+1}||_{\mathcal{D}_{\nu+1}}+||g_{\nu+1}||_{\mathcal{D}_{\nu+1}}\leq \gamma^{n+m+2}_{0}s^{m}_{\nu+1}\mu_{\nu+1}.
\end{equation*}
\end{lemma}
\begin{proof}
Note that $\bar{f}_{\nu+1}$ and $\bar{g}_{\nu+1}$ are solved by the implicit function theorem from \eqref{eq-4.10} and \eqref{eq-4.11}.
Thus,
\begin{align}\label{equation-4.38}
||\bar{f}_{\nu+1}||_{\mathcal{D}_{\nu+1}}&\leq c ||\partial_{\theta_{\nu+1}}f_{\nu}||_{\hat{\mathcal{D}}_{\nu+1}}||U_{\nu+1}||_{\mathcal{D}_{\nu+1}}+c||\partial_{r_{\nu+1}}f_{\nu}||_{\hat{\mathcal{D}}_{\nu+1}}||V_{\nu+1}||_{\mathcal{D}_{\nu+1}}\notag\\
&+c||\mathcal{R}_{K_{\nu+1}}f_{\nu}||_{\mathcal{D}_{\nu+1}}+c\varpi_{*}(|V_{\nu+1}|),
\end{align}
where $\varpi_{1}(|V_{\nu+1}|)\leq \varpi_{*}(|V_{\nu+1}|)$ is due to $\varpi_{1}\leq \varpi_{*}$.

The intersection property implies that there exists a $\theta^{0}_{\nu+1}$ such that for each $r^{0}_{\nu+1}\in G(s_{\nu+1})$, one has
\begin{equation*}
\bar{g}_{\nu+1}(\theta^{0}_{\nu+1},r^{0}_{\nu+1},\e)=0,
\end{equation*}
i.e.,
\begin{align*}
&\sup_{\theta_{\nu+1}\in D(h_{\nu+1})}||\bar{g}_{\nu+1}(\theta_{\nu+1},r_{\nu+1},\e)||\\
=&\sup_{\theta_{\nu+1}\in D(h_{\nu+1})}||\bar{g}_{\nu+1}(\theta_{\nu+1},r_{\nu+1},\e)-\bar{g}_{\nu+1}(\theta^{0}_{\nu+1},r^{0}_{\nu+1},\e)||\\
=&\mathop{{\rm osc}}_{\theta_{\nu+1}\in D(h_{\nu+1})}\bar{g}_{\nu+1}(\theta_{\nu+1},r_{\nu+1},\e)\\
=&\mathop{{\rm osc}}_{\theta_{\nu+1}\in D(h_{\nu+1})}(\bar{g}_{\nu+1}(\theta_{\nu+1},r_{\nu+1},\e)-\bar{h}),
\end{align*}
where $\bar{h}$ is a function of $r_{\nu+1}$. Specially, taking $\bar{h}=g_{0,\nu}(r_{\nu+1})$, one has
\begin{equation*}
\frac{1}{2}||\bar{g}_{\nu+1}||_{\mathcal{D}_{\nu+1}}\leq ||\bar{g}_{\nu+1}-g_{0,\nu}||_{\mathcal{D}_{\nu+1}}.
\end{equation*}
Therefore,
\begin{align}\label{equation-4.40}
||\bar{g}_{\nu+1}||_{\mathcal{D}_{\nu+1}}&\leq c||\bar{g}_{\nu+1}-g_{0,\nu}||_{\mathcal{D}_{\nu+1}}\notag\\
&\leq c||\partial_{\theta_{\nu+1}}g_{\nu}||_{\hat{\mathcal{D}}_{\nu+1}}||U_{\nu+1}||_{\mathcal{D}_{\nu+1}}+c||\partial_{r_{\nu+1}}g_{\nu}||_{\hat{\mathcal{D}}_{\nu+1}}||V_{\nu+1}||_{\mathcal{D}_{\nu+1}}\notag\\
&+c||\mathcal{R}_{K_{\nu+1}}g_{\nu}||_{\mathcal{D}_{\nu+1}}.
\end{align}
Following \eqref{equation-4.38}, \eqref{equation-4.40}, $({\rm H4})$, and estimates obtained earlier, we have
\begin{align*}
||\bar{f}_{\nu+1}||_{\mathcal{D}_{\nu+1}}+||\bar{g}_{\nu+1}||_{\mathcal{D}_{\nu+1}}&\leq c||\partial_{\theta_{\nu+1}}f_{\nu}||_{\hat{\mathcal{D}}_{\nu+1}}||U_{\nu+1}||_{\mathcal{D}_{\nu+1}}+c||\partial_{\theta_{\nu+1}}g_{\nu}||_{\hat{\mathcal{D}}_{\nu+1}}||U_{\nu+1}||_{\mathcal{D}_{\nu+1}}\notag\\
&+ c||\partial_{r_{\nu+1}}f_{\nu}||_{\hat{\mathcal{D}}_{\nu+1}}||V_{\nu+1}||_{\mathcal{D}_{\nu+1}}+c||\partial_{r_{\nu+1}}g_{\nu}||_{\hat{\mathcal{D}}_{\nu+1}}||V_{\nu+1}||_{\mathcal{D}_{\nu+1}}\notag\\
&+ ||\mathcal{R}_{K_{\nu+1}}f_{\nu}||_{\mathcal{D}_{\nu+1}}+c||\mathcal{R}_{K_{\nu+1}}g_{\nu}||_{\mathcal{D}_{\nu+1}}+c\varpi_{*}(|V_{\nu+1}|)\notag\\
&\leq c\frac{\gamma^{n+m+2}_{0}s^{m}_{\nu}\mu_{\nu}}{h_{\nu+1}-h_{\nu+2}}\cdot\gamma^{n+m+1}_{0}s^{m}_{\nu}\mu_{\nu}\Gamma(h_{\nu}-h_{\nu+1})\notag\\
&+ c\frac{\gamma^{n+m+2}_{0}s^{m}_{\nu}\mu_{\nu}}{s_{\nu+1}-s_{\nu+2}}\cdot\gamma^{n+m+1}_{0}s^{m}_{\nu}\mu_{\nu}\Gamma(h_{\nu}-h_{\nu+1})\notag\\
&+ c\gamma^{n+m+2}_{0}s^{m}_{\nu}\mu^{2}_{\nu}\notag\\
&\leq c_{4}\gamma^{n+m+2}_{0}s^{m}_{\nu}\mu^{2}_{\nu}\cdot\gamma^{n+m+1}_{0}s^{m}_{\nu}(h_{\nu+1}-h_{\nu+2})^{-1}\Gamma(h_{\nu}-h_{\nu+1})\\
&+ c_{4}\gamma^{n+m+2}_{0}s^{m}_{\nu}\mu^{2}_{\nu}\cdot\gamma^{n+m+1}_{0}s^{m-1}_{\nu}\Gamma(h_{\nu}-h_{\nu+1})\\
&+ c_{4}\gamma^{n+m+2}_{0}s^{m}_{\nu}\mu^{2}_{\nu}.
\end{align*}
 Finally, $(\rm H5)$ implies that
 \begin{equation*}
 ||f_{\nu+1}||_{\mathcal{D}_{\nu+1}}+||g_{\nu+1}||_{\mathcal{D}_{\nu+1}}\leq \gamma^{n+m+2}_{0}s^{m}_{\nu+1}\mu_{\nu+1}.
 \end{equation*}
 \end{proof}
 \subsection{The preservation of intersection property}
 In the previous Section \ref{subsec-4.2.3}, we have constructed a translation $\mathscr{V}_{\nu+1}$ such that the frequency $\omega_{0}(r_{0})$ unchanged. The translation $\mathscr{V}_{\nu+1}$ truns $\bar{\mathscr{F}}_{\nu+1}$ into $\mathscr{F}_{\nu+1}=\bar{\mathscr{F}}_{\nu+1}\circ\mathscr{V}_{\nu+1}$ but drops the intersection property. For this purpose, we  construct the conjugation of $ \bar {\mathscr{F}}_{\nu+1} $ such that it has the same properties as $ \bar {\mathscr{F}}_{\nu+1} $. Denote by $\hat{\mathscr{F}}_{\nu+1}$ the conjugation of $\bar{\mathscr{F}}_{\nu+1}$, that is, $\hat{\mathscr{F}}_{\nu+1}=\mathscr{V}^{-1}_{\nu+1}\circ\bar{\mathscr{F}}_{\nu+1}\circ\mathscr{V}_{\nu+1}$, where
 \begin{equation*}
 \mathscr{V}^{-1}_{\nu+1}: \theta^{1}_{\nu+1}\rightarrow \theta^{1}_{\nu+1},\quad \hat{r}^{1}_{\nu+1}\rightarrow \hat{r}^{1}_{\nu+1}-r^{*}_{\nu+1}.
 \end{equation*}
 Therefore, $\hat{\mathscr{F}}_{\nu+1}$ has the form
\begin{equation*}\hat{\mathscr{F}}_{\nu+1}:
\begin{cases}
\theta^{1}_{\nu+1}=\theta_{\nu+1}+\omega_{0}(r_{0})+f_{\nu+1}(\theta_{\nu+1},\hat{r}_{\nu+1},\e),\\
\hat{r}^{1}_{\nu+1}=\hat{r}_{\nu+1}-\sum\limits^{\nu+1}_{i=0}r^{*}_{i}+g_{\nu+1}(\theta_{\nu+1},\hat{r}_{\nu+1},\e).
\end{cases}
\end{equation*}
It ensures that the mapping $\hat{\mathscr{F}}_{\nu+1}$ still has the intersection property.

\section{Proof of the main results}\label{SEC5}
In this section, we will show the proof of Theorem \ref{theorem-1}, Corollary \ref{cor-1} and Theorem  \ref{theorem-2} successively.
\subsection{Proof of Theorem \ref{theorem-1}}
\setcounter{equation}{0}
\subsubsection{Iteration lemma}
The iteration lemma guarantees the inductive construction of the transformations in all KAM steps. Let $s_{0}$, $h_{0}$, $\gamma_{0}$, $\mu_{0}$, $\mathscr{F}_{0}$, $\mathcal{D}_{0}$ be given in Section \ref{section-2}, and set $K_{0}=0$, $r^{*}_{0}=0$, $0<\rho<1$ is a constant. We define the following sequences inductively for all $\nu=0,1,2,\cdots$.
\begin{align*}
h_{\nu+1}&=\frac{h_{\nu}}{2}+\frac{h_{0}}{4},\\
s_{\nu+1}&=\frac{s_{\nu}}{2},\\
\mu_{\nu+1}&=\mu^{1+\rho}_{\nu},\\
K_{\nu+1}&=([\log\frac{1}{\mu_{\nu}}]+1)^{3\eta},\\
\mathcal{D}_{\nu+1}&=\mathcal{D}(h_{\nu+1},s_{\nu+1}),\\
\hat{\mathcal{D}}_{\nu+1}&=\mathcal{D}(h_{\nu+2}+\frac{3}{4}(h_{\nu+1}-h_{\nu+2}),s_{\nu+2}),\\
\Gamma(h_{\nu}-h_{\nu+1})&=\sum_{0<|k|\leq K_{\nu+1}}|k|^{\tau}e^{-|k|\frac{h_{\nu}-h_{\nu+1}}{4}}\leq \frac{4^{\tau}\tau!}{(h_{\nu}-h_{\nu+1})^{\tau}}.
\end{align*}
\begin{lemma}
Consider mapping \eqref{equation-1} for $\nu=0,1,2,\cdots$. If $\e_{0}$ is sufficiently small such that $(\rm H1)-(\rm \rm H5)$ hold, and
\begin{equation*}
||f_{\nu}||_{\mathcal{D}_{\nu}}+||g_{\nu}||_{\mathcal{D}_{\nu}}\leq \gamma^{n+m+2}_{0}s^{m}_{\nu}\mu_{\nu},
\end{equation*}
then the iteration process described above is valid, and the following properties hold.
\begin{itemize}
\item[$(\rm i)$]
There exists a real analytic transformation  $\mathscr{W}_{\nu+1}:=\mathscr{U}_{\nu+1}\circ\mathscr{V}_{\nu+1}$ that satisfies $\mathscr{W}_{\nu+1}\circ\hat{\mathscr{F}}_{\nu+1}=\hat{\mathscr{F}}_{\nu}\circ\mathscr{W}_{\nu+1}$, where
\begin{equation*}\hat{\mathscr{F}}_{\nu+1}:
\begin{cases}
\theta^{1}_{\nu+1}=\theta_{\nu+1}+\omega_{0}(r_{0})+f_{\nu+1}(\theta_{\nu+1},\hat{r}_{\nu+1},\e),\\
\hat{r}^{1}_{\nu+1}=\hat{r}_{\nu+1}-\sum\limits^{\nu+1}_{i=0}r^{*}_{i}+g_{\nu+1}(\theta_{\nu+1},\hat{r}_{\nu+1},\e).
\end{cases}
\end{equation*}

Also, the transformation $\mathscr{W}_{\nu+1}$ has the estimate
\begin{equation}\label{equation-51}
||\mathscr{W}_{\nu+1}-{\rm id}||_{\mathcal{D}_{\nu+1}}\leq c_{3}\gamma^{n+m+1}_{0}s^{m}_{\nu}\mu_{\nu}\Gamma(h_{\nu}-h_{\nu+1}).
\end{equation}
\item[$(\rm ii)$] $\{\hat{r}_{\nu}\}$ is a Cauchy sequence and
\begin{equation*}\label{equation-5.1}
|\hat{r}_{\nu+1}-\hat{r}_{\nu}|\leq c\mu_{\nu}.
\end{equation*}
\item[$(\rm iii)$] The estimate on new perturbations is
\begin{equation*}\label{equation-5.2}
||f_{\nu+1}||_{\mathcal{D}_{\nu+1}}+||g_{\nu+1}||_{\mathcal{D}_{\nu+1}}\leq \gamma^{n+m+2}_{0}s^{m}_{\nu+1}\mu_{\nu+1}.
\end{equation*}
\end{itemize}
\end{lemma}
\begin{proof}
The proof is an induction on $\nu$. It is easy to see that we can take sufficiently small $\e_{0}$ to ensure that $(\rm H1)-(\rm H5)$ hold. Since
\begin{align*}
\hat{\mathscr{F}}_{\nu+1}&=\mathscr{V}^{-1}_{\nu+1}\circ\bar{\mathscr{F}}_{\nu+1}\circ\mathscr{V}_{\nu+1}\notag\\
&=\mathscr{V}^{-1}_{\nu+1}\circ(\mathscr{U}^{-1}_{\nu+1}\circ\hat{\mathscr{F}}_{\nu}\circ\mathscr{U}_{\nu+1})\circ\mathscr{V}_{\nu+1},
\end{align*}
and $\mathscr{W}_{\nu+1}=\mathscr{U}_{\nu+1}\circ\mathscr{V}_{\nu+1}$, we have that $\mathscr{W}_{\nu+1}\circ\hat{\mathscr{F}}_{\nu+1}=\hat{\mathscr{F}}_{\nu}\circ\mathscr{W}_{\nu+1}$. For the estimate \eqref{equation-5.1} in $(\rm i)$, see Lemma \ref{lemma-4.6}. In addition, we notice that $(\rm ii)$ is due to \eqref{equation-4.33}, and $({\rm iii})$ follows from Lemma \ref{lemma-4.7}.
\end{proof}
\subsubsection{Convergence}
Observe that
\begin{align}\label{equation-55}
\hat{\mathscr{F}}_{\nu+1}&=\mathscr{W}^{-1}_{\nu+1}\circ\hat{\mathscr{F}}_{\nu}\circ\mathscr{W}_{\nu+1}\notag\\
&=\mathscr{W}^{-1}_{\nu+1}\circ\mathscr{W}^{-1}_{\nu}\circ\hat{\mathscr{F}}_{\nu-1}\circ\mathscr{W}_{\nu}\circ\mathscr{W}_{\nu+1}\notag\\
&=\cdots\notag\\
&=\mathscr{W}^{-1}_{\nu+1}\circ\cdots\circ\mathscr{W}^{-1}_{1}\circ\mathscr{F}_{0}\circ\mathscr{W}_{1}\circ\cdots\circ\mathscr{W}_{\nu+1}.
\end{align}
 Denote
\begin{align*}
 \mathscr{W}^{\nu+1} &:= \mathscr{W}_{1}\circ\mathscr{W}_{2}\circ\cdots\circ\mathscr{W}_{\nu+1},
 \end{align*}
 then \eqref{equation-55} implies that
   \begin{equation*}
  \mathscr{W}^{\nu+1}\circ\hat{\mathscr{F}}_{\nu+1}=\mathscr{F}_{0}\circ\mathscr{W}^{\nu+1}.
  \end{equation*}
  The transformation $\mathscr{W}^{\nu+1}$ is convergent since
  \begin{align}
|| \mathscr{W}^{\nu+1}-\mathscr{W}^{\nu}||_{\mathcal{D}_{\nu+1}}&=||\mathscr{W}_{1}\circ\cdots\circ\mathscr{W}_{\nu}\circ\mathscr{W}_{\nu+1}-\mathscr{W}_{1}\circ\cdots\circ\mathscr{W}_{\nu}||_{\mathcal{D}_{\nu+1}}\notag\\
&\leq||\mathscr{W}^{\nu}||_{\mathcal{D}_{\nu}}||\mathscr{W}_{\nu+1}-{\rm id}||_{\mathcal{D}_{\nu+1}}\notag\\
&\leq\prod\limits^{\nu}_{i=1}(1+c\gamma^{n+m+1}_{0}s^{m}_{i-1}\mu_{i-1}\Gamma(h_{i-1}-h_{i}))c\gamma^{n+m+1}_{0}s^{m}_{\nu}\mu_{\nu}\Gamma(h_{\nu}-h_{\nu+1})\notag\\
&\leq c\gamma^{n+m+1}_{0}s^{m}_{\nu}\mu_{\nu}\Gamma(h_{\nu}-h_{\nu+1}).
  \end{align}
   Therefore, $\lim\limits_{\nu\rightarrow \infty}\mathscr{W}^{\nu}:=\mathscr{W}$, as well as $\mathscr{F}_{\infty}=\lim\limits_{\nu\rightarrow\infty}\hat{\mathscr{F}}_{\nu}$, we thus deduce that
\begin{equation*}
\mathscr{W}\circ\mathscr{F}_{\infty}=\mathscr{F}\circ\mathscr{W}.
\end{equation*}

It remains to consider the following convergence. By Lemma \ref{lemma-4.5}, one has
\begin{align}
&\omega_{0}(\hat{r}_{1})+f_{0,0}(\hat{r}_{1})=\omega_{0}(r_{0}),\notag\\
&\omega_{0}(\hat{r}_{2})+f_{0,0}(\hat{r}_{2})+f_{0,1}(\hat{r}_{2})=\omega_{0}(r_{0}),\notag\\
&~~~~~~~~~~~~~~~~~~~~~~~~~~~\vdots\notag\\
\label{equation-5.48}
&\omega_{0}(\hat{r}_{\nu})+f_{0,0}(\hat{r}_{\nu})+\cdots+f_{0,\nu-1}(\hat{r}_{\nu})=\omega_{0}(r_{0}).
\end{align}
Taking limits on both sides of \eqref{equation-5.48}, we obtain
\begin{equation*}
\omega_{0}(\hat{r}_{\infty})+\sum^{\infty}_{i=0}f_{0,i}(\hat{r}_{\infty})=\omega_{0}(r_{0})=\omega(r_{*}),
\end{equation*}
that is, for given $r_{*}\in E^{\circ}$, the mapping $\mathscr{F}_{\infty}$ on $\mathcal{D}_{\infty}$ becomes the integrable rotation
\begin{equation*}\mathscr{F}_{\infty}:
\begin{cases}
\theta^{1}_{\infty}=\theta_{\infty}+\omega(r_{*}),\\
r^{1}_{\infty}=r_{\infty}-\tilde{r},
\end{cases}
\end{equation*}
where $\omega(r_{*})=p$, and $ p $ is given in advance. Besides,  $\tilde{r}=\sum\limits^{\infty}_{i=0}r^{*}_{i} \to 0$ as $ \varepsilon \to 0 $. This completes the proof of the  frequency-preserving KAM persistence in Theorem \ref{theorem-1}.

\subsection{Proof of Corollary \ref{cor-1}}
This subsection is devoted to the proof of  Corollary \ref{cor-1}. We will show that the assumption on $ \omega(r) $ here contains the transversality condition $({\rm A1})$. In fact, the frequency mapping $\omega(r)$ is injective on $E^{\circ}$, and consequently, it is surjective from $E^{\circ}$ to $\omega({E^{\circ}})$.  Therefore it is a homeomorphism and by Nagumo's theorem, we have that the Brouwer degree $\deg(\omega,E^{\circ},p)=\pm 1$ for some $ p \in \omega(E^{\circ})^\circ $. Finally, by applying Theorem \ref{theorem-1} we directly obtain the desired frequency-preserving KAM persistence in Corollary \ref{cor-1}.

\subsection{Proof of Theorem \ref{theorem-2}}
This section outlines the proof of Theorem \ref{theorem-2}. We focus on describing the parts of the proof of Theorem \ref{theorem-2}  that differ from those of Theorem \ref{theorem-1}. We shall see that $\varpi_{1}\leq \varpi_{2}$ is used directly by assumption $({\rm A3})$ in the process of proving Lemma \ref{lemma-5.2}. Moreover, there is no need to construct a conjugation of $\bar{\mathscr{F}}_{\nu+1}$ since the mapping we considered in Theorem \ref{theorem-2} does not have the intersection property. The parts we leave out in this section are similar to those described in Section \ref{SEC4}.

Consider the mapping $\mathscr{F}:\mathbb{T}^{n}\times \Lambda\rightarrow \mathbb{T}^{n}$ defined by
\begin{equation*}
\theta^{1}=\theta+\omega(\xi)+\e f(\theta,\xi,\e),
\end{equation*}
where $\xi\in\Lambda\subset \mathbb{R}^{n}$ is a parameter, $\Lambda$ is a connected closed bounded domain with interior points.  For $\theta_{0}\in D(h_{0})$, and $\xi_{0}\in\Lambda_{0}:=\{\xi\in\Lambda|\ |\xi-\xi_{0}|<{\rm dist}\ (\xi_{0},\partial \Lambda)\}$, denote
\begin{equation*}
\theta^{1}_{0}=\theta_{0}+\omega(\xi_{0})+f_{0}(\theta_{0},\xi_{0},\e),
\end{equation*}
where $\omega(\xi_{0})=\omega(\xi_{*})=q$ for given $ q\in \mathbb{R}^n $ in advance and $f_{0}(\theta_{0},\xi_{0},\e)=\e f(\theta_{0},\xi_{0},\e)$. The estimate on $||f_{0}||_{D(h_{0})}$ is
\begin{equation*}
||f_{0}||_{D(h_{0})}\leq \gamma^{n+m+2}_{0}\mu_{0},
\end{equation*}
if
\begin{equation*}
\e^{\frac{3}{4}}\e^{-\frac{1}{8\eta(m+1)}}||f||_{D(h_{0})}\leq 1.
\end{equation*}
Set
\begin{equation*}
\Lambda_{\nu}:=\{\xi:{\rm dist}\ (\xi,\partial\Lambda_{\nu-1}) <\mu_{\nu-1}\}, \;\; \nu \in \mathbb{N}^+.
\end{equation*}
Suppose that after $\nu$ KAM steps, for $\theta_{\nu}\in D(h_{\nu})$ and $\xi_{\nu}\in\Lambda_{\nu}$, the mapping becomes
\begin{equation*}
\theta^{1}_{\nu}=\theta_{\nu}+\omega(\xi_{0})+f_{\nu}(\theta_{\nu},\xi_{\nu},\e),
\end{equation*}
and one has
\begin{equation*}
||f_{\nu}||_{D(h_{\nu})}\leq \gamma^{n+m+2}_{0}\mu_{\nu}.
\end{equation*}
Introduce a transformation $\mathscr{U}_{\nu+1}:={\rm id}+U_{\nu+1}$ that satisfies $\mathscr{U}_{\nu+1}\circ\bar{\mathscr{F}}_{\nu+1}=\mathscr{F}_{\nu}\circ\mathscr{U}_{\nu+1}$. Then the conjugation $\bar{\mathscr{F}}_{\nu+1}$ of mapping $\mathscr{F}_{\nu}$ is
\begin{equation*}
\bar{\mathscr{F}}_{\nu+1}: \theta^{1}_{\nu+1}=\theta_{\nu+1}+\omega(\xi_{0})+\bar{f}_{\nu+1}(\theta_{\nu+1},\xi_{\nu},\e).
\end{equation*}
We obtain the homological equation
\begin{equation}\label{eq-56}
U_{\nu+1}(\theta_{\nu+1}+\omega(\xi_{0}))-U_{\nu+1}(\theta_{\nu+1})=\mathcal{T}_{K_{\nu+1}}f_{\nu}(\theta_{\nu+1},\xi_{\nu},\e),
\end{equation}
and the new perturbation
\begin{align*}
\bar{f}_{\nu+1}(\theta_{\nu+1},\xi_{\nu},\e)&=f_{\nu}(\theta_{\nu+1}+U_{\nu+1},\xi_{\nu},\e)-f_{\nu}(\theta_{\nu+1},\xi_{\nu},\e)+\mathcal{R}_{K_{\nu+1}}f_{\nu}(\theta_{\nu+1},\xi_{\nu},\e)\\
&+U_{\nu+1}(\theta_{\nu+1}+\omega(\xi_{0}),\xi_{\nu},\e)-U_{\nu+1}(\theta_{\nu+1}+\omega(\xi_{0})+\bar{f}_{\nu+1},\xi_{\nu},\e).
\end{align*}
The homological equation \eqref{eq-56} is uniquely solvable on $D(h_{\nu+1})$, and the new perturbation $\bar{f}_{\nu+1}$ can be solved by the implicit function theorem.

To keep the frequency unchanged, construct a translation
	\begin{equation}\label{5.6}
	\mathscr{V}_{\nu+1}: \theta_{\nu+1}\rightarrow \theta_{\nu+1},\quad \tilde{\xi}_{\nu}\rightarrow \tilde{\xi}_{\nu}+\xi_{\nu+1}-\xi_{\nu},
\end{equation}
where $\xi_{\nu+1}$ is to be determined. This translation changes the parameter alone, and the mapping becomes $\mathscr{F}_{\nu+1}=\bar{\mathscr{F}}_{\nu+1}\circ\mathscr{V}_{\nu+1}$, that is,
\begin{equation*}
	\mathscr{F}_{\nu+1}: \theta^{1}_{\nu+1}=\theta_{\nu+1}+\omega(\xi_{0})+f_{\nu+1}(\theta_{\nu+1},\xi_{\nu+1},\e),
\end{equation*}
where	the frequency $\omega(\xi_{0})=\omega(\xi_{\nu+1})+\sum\limits^{\nu}_{i=0}f_{0,i}(\xi_{\nu+1})$. The following lemma states that the frequency is preserved.

 \begin{lemma}\label{lemma-5.2}
 Assume that
 \begin{equation*}
 ({\rm H6})\qquad ||\sum^{\nu}_{i=0}f_{0,i}(\xi_{\nu})||_{\Lambda_{\nu}}\leq c\mu^{\frac{1}{2}}_{0}.
 \end{equation*}
 Then there exists a $\xi_{\nu+1}\in B_{c\mu_{\nu}}(\xi_{\nu})\subset \Lambda_{0}$ such that
 \begin{equation*}
 \omega(\xi_{\nu+1})+\sum^{\nu}_{i=0}f_{0,i}(\xi_{\nu+1})=\omega(\xi_{0}).
 \end{equation*}
 \end{lemma}
 \begin{proof}The proof is an induction on $\nu\in\mathbb{N}$. When $\nu=0$, obviously, $\omega(\xi_{0})=\omega(\xi_{0})$. When $\nu\geq 1$, let
 \begin{equation}\label{eq-5.6}
 \omega(\xi_{j})+\sum^{\nu-1}_{i=0}f_{0,i}(\xi_{j})=\omega(\xi_{0}),\qquad j=1,2,\cdots,\nu,
 \end{equation}
 then
 \begin{equation}\label{eq-5.7}
 \omega(\xi_{\nu+1})+\sum^{\nu}_{i=0}f_{0,i}(\xi_{\nu+1})=\omega(\xi_{0})
 \end{equation}
needs to be verified. Taking the assumptions $({\rm B1})$ and $({\rm B3})$, one has
\begin{equation}\label{DISA11}
\deg\left(\omega(\xi_{\nu+1})+\sum^{\nu}_{i=0}f_{0,i}(\xi_{\nu+1}),\Lambda_{0},\omega(\xi_{0})\right)=\deg\left(\omega(\xi_{\nu+1}),\Lambda_{0},\omega(\xi_{0})\right)\neq0.
\end{equation}
This means that there exists at least one parameter $\xi_{\nu+1}\in\Lambda_{0}$ such that $\omega(\xi_{\nu+1})+\sum\limits^{\nu}_{i=0}f_{0,i}(\xi_{\nu+1})=\omega(\xi_{0})$ holds.
Next, let us verify that $\xi_{\nu+1}\in B_{c\mu_{\nu}}(\xi_{\nu})\subset \Lambda_{\nu}$. From \eqref{eq-5.6} and \eqref{eq-5.7}, one has
\begin{equation*}
\omega(\xi_{\nu})+\sum^{\nu-1}_{i=0}f_{0,i}(\xi_{\nu})=\omega(\xi_{\nu+1})+\sum^{\nu}_{i=0}f_{0,i}(\xi_{\nu+1}),
\end{equation*}
i.e.,
\begin{equation*}
f_{0,\nu}(\xi_{\nu+1})=\omega(\xi_{\nu})-\omega(\xi_{\nu+1})+\sum^{\nu-1}_{i=0}(f_{0,i}(\xi_{\nu})-f_{i}(\xi_{\nu+1})).
\end{equation*}
Since $||f_{i}||_{D(h_{i+1})}\leq \gamma^{n+m+2}_{0}\mu_{i}$ for each $ i\geq 0$, one has $[f_{0,i}]_{\varpi_{1}}\leq c\mu_{i}$. Therefore,
\begin{equation*}
| f_{0,i}(\xi_{\nu})-f_{0,i}(\xi_{\nu+1})|\leq c\mu_{i}\varpi_{1}(| \xi_{\nu}-\xi_{\nu+1}|).
\end{equation*}
Following $({\rm A3})$, we have
\begin{align*}
| f_{0,\nu}(\xi_{\nu+1})|&=|\omega(\xi_{\nu})-\omega(\xi_{\nu+1})+\sum^{\nu-1}_{i=0}(f_{0,i}(\xi_{\nu})-f_{0,i}(\xi_{\nu+1}))|\\
&\geq|\omega(\xi_{\nu})-\omega(\xi_{\nu+1})|-\sum^{\nu-1}_{i=0}|f_{0,i}(\xi_{\nu})-f_{0,i}(\xi_{\nu+1})|\\
&\geq\varpi_{2}(| \xi_{\nu}-\xi_{\nu+1}|)-c\varpi_{1}(|\xi_{\nu}-\xi_{\nu+1}|)\sum^{\nu-1}_{i=0}\mu_{i}\\
&\geq\frac{\varpi_{2}(|\xi_{\nu}-\xi_{\nu+1}|)}{2}.
\end{align*}
The last inequality is due to $\varpi_{1}\leq \varpi_{2}$, and $\e$ is sufficiently small such that $c(\sum\limits^{\nu-1}_{i=0}\mu_{i})\leq \frac{1}{2}$. Thus,
\begin{equation}\label{2.8}
|\xi_{\nu}-\xi_{\nu+1}|\leq \varpi^{-1}_{2}(2|f_{0,\nu}(\xi_{\nu+1})|)\leq \varpi^{-1}_{2}(2c \mu_{\nu})\leq \varpi^{-1}_{1}(2c\mu_{\nu})\leq c\mu_{\nu},
\end{equation}
which is similar to \eqref{equation-4.33}. Moreover,
\begin{equation*}
| \xi_{\nu+1}-\xi_{0}|\leq\sum^{\nu}_{i=0}| \xi_{i+1}-\xi_{i}|\leq c\sum^{\nu}_{i=0}\mu_{i}\leq 2c\mu_{0}.
\end{equation*}
The above illustrates that $\{\xi_{\nu}\}$ is a Cauchy sequence and $\xi_{\nu+1}\in B_{c\mu_{\nu}}(\xi_{\nu})\subset \Lambda_{0}$.
 \end{proof}

We now summarize the standard convergence. Denote
\[\mathscr{U}^{\nu+1}:=\mathscr{U}_{1}\circ\cdots\circ\mathscr{U}_{\nu+1},\]
and
\[\mathscr{W}^{\nu+1}:=\mathscr{W}_{1}\circ\cdots\circ\mathscr{W}_{\nu+1}:=(\mathscr{U}_{1}\circ\mathscr{V}_{1})\circ\cdots\circ(\mathscr{U}_{\nu+1}\circ\mathscr{V}_{\nu+1}).\]
Both of them are convergent,  one therefore has $\mathscr{U}:=\lim\limits_{\nu\rightarrow\infty}\mathscr{U}^{\nu}$ and $\mathscr{W}:=\lim\limits_{\nu\rightarrow\infty}\mathscr{W}^{\nu}$. Moreover, we get
\[\omega(\xi_{\infty})+\sum\limits^{\infty}_{i=0}f_{0,i}(\xi_{\infty})=\omega(\xi_{0})=\omega(\xi_{*})=q\]
for $\xi_{*}\in \Lambda^{\circ}$ fixed. For \eqref{UUUU}, we know that the conjugation for dynamical system \eqref{equation-3.1} only focuses on the angular variable $ \theta \in \mathbb{T}^n $. As mentioned in \eqref{5.6},  the angular variable is unchanged under the translation. Therefore, one has
\begin{equation*}
	\mathscr{U}\circ\mathscr{F}_{\infty}=\mathscr{F}\circ\mathscr{U},
\end{equation*}
here $\mathscr{F}_{\infty}$ denotes the limit of $\mathscr{F}_{\nu}$ on $\mathcal{D}_{\infty}$. More precisely, one obtains the integrable rotation
\begin{equation*}
	\mathscr{F}_{\infty}:\theta^{1}_{\infty}=\theta_{\infty}+\omega(\xi_{*})
\end{equation*}
with frequency $ \omega(\xi_{*})=q $, where $ q $ is given in advance. In other words, we prove the KAM persistence with prescribed frequency-preserving. This completes the proof of Theorem \ref{theorem-2}.


 \section*{Acknowledgements}
This work was supported by National Basic Research Program of China (grant No. 2013CB834100), National Natural Science Foundation of China (grant No. 11571065, 11171132, 12071175), Project of Science and Technology Development of Jilin Province, China (grant No. 2017C028-1, 20190201302JC), and Natural Science Foundation of Jilin Province (grant No. 20200201253JC).

\end{document}